\documentclass[11pt]{article}
\usepackage{amsgen, amsmath, amsfonts,amsthm, amssymb, enumerate,amscd}
\usepackage[all]{xy}
\input xy
\xyoption{all}

\topmargin -.15in \textheight 9in \numberwithin{equation}{section}

\newtheorem{defn}{Definition}[section]
\newtheorem{prop}[defn]{Proposition}
\newtheorem{lem}[defn]{Lemma}
\newtheorem{thm}[defn]{Theorem}
\newtheorem{cor}[defn]{Corollary}

\newtheorem{conj}[defn]{Conjecture}
\newcommand {\cf}{{\it cfr.}}
\newcommand {\ZZ}{{\mathbb Z}}

\newcommand {\XX}{{\mathcal X}}
\newcommand {\K}{{\mathcal K}}

\newcommand {\C}{{\mathbb C}}

\newcommand {\F}{{\mathbb F}}

\newcommand {\Q}{{\mathbb Q}}
\newcommand {\R}{{\mathbb R}}
\newcommand {\OO}{{\mathcal O}}

\newcommand {\M}{{\mathcal M}}
\newcommand {\m}{{\mathfrak m}}

\newcommand {\PP}{{\mathfrak p}}

\newcommand {\CP}{{\mathbb P}}
\newcommand {\T}{{\cal T}}
\newcommand {\KI}{{K_{\infty}}}
\newcommand {\I}{{\mathcal I}}
\newcommand {\CI}{{{\mathbb C}_{\infty}}}

\newcommand {\D}{{\mathcal D}}
\newcommand {\SSP}{{\mathfrak sp}}
\newcommand {\bi}{{\mathbf i}}

\newcommand {\Z}{{\mathcal Z}}
\newcommand {\BP}{{\mathbf P}}

\def\sgn{\operatorname{sgn}}

\def\div{\operatorname{div}}
\def\deg{\operatorname{deg}}

\def\dim{\operatorname{dim}}

\def\ord{\operatorname{ord}}

\title{$K_1$ of products of Drinfeld modular curves and
special values of $L$-functions}
\author{Ramesh  Sreekantan}

\begin{document}
\maketitle

\begin{abstract}
In \cite{beil} Beilinson obtained a formula relating the special
value of the $L$-function of $H^2$ of a product of modular curves to
the regulator of an element of a motivic cohomology group - thus
providing evidence for his general conjectures on special values of
$L$-functions. In this paper we prove a similar formula for the
$L$-function of the product of two Drinfeld modular curves providing
evidence for an analogous conjecture in the case of function fields.
\end{abstract}


\vspace{.1in}

\section{Introduction}

\subsection{Beilinson's conjectures and a function field analogue}

The  algebraic $K$-theory  of  a smooth projective variety  over  a
field has a finite, increasing filtration called the  Adams
filtration. For a variety over a  number field, in \cite{beil}
Beilinson formulated conjectures which relate the graded  pieces of
this filtration, the motivic cohomology groups $H^\ast_\mathcal M$,
to special values of the Hasse-Weil $L$-function of a cohomology
group of the variety.

The conjectures are of the following nature: corresponding to the
motivic cohomology group $H^*_{\M}$ there is a real vector space
$H^*_{\D}$, called the real Deligne cohomology, whose dimension is
the order of the pole, at a specific point,  of the Archimedean
factor of the $L$-function.

Beilinson defined a regulator map from the $H^*_{\M}$ to $H^*_{\D}$
and conjectured that its image determines a $\Q$-structure on the
$H^*_{\D}$. $H^*_{\D}$ has another $\Q$-structure induced by de Rham
and Betti cohomology groups. Beilinson conjectured further that the
determinant of the change of basis between these two $\Q$ structures
is, up to a non-zero rational number, the first non-zero term in the
Taylor expansion of the $L$-function at a specific point. More
details can be found in the book \cite{rss} or in the paper
\cite{rama}.

Beilinson's conjectures have been  proved only in a few special
cases. In \cite{beil}, he proved them for  the product of two
modular curves and as a result for the product of two non-isogenous
elliptic curves over $\Q$. It is these results that we generalize to
the function field case.

Since the conjectures deal with the transcendental part of the value
of the $L$-function and involve the Archimedean $L$-factor they can
be viewed as conjectures for the Archimedean place. It is natural to
ask whether one can formulate a similar question for the other
finite places.

In \cite{sree}, we formulated a function field analogue of the
Beilinson conjectures. In particular we defined a group which, at a
finite place, plays the role of the real Deligne cohomology.

This group, called the $\nu$-adic Deligne cohomology, is a rational
vector space whose dimension  was shown by Consani \cite{cons1},
assuming some standard conjectures, to coincide with the order of
the pole, at a certain integer, of the local $L$-factor at the place
$\nu$.

In \cite{sree} we defined  a regulator map $r_{ \D,\nu}$ from the
motivic cohomology  to the $\nu$-adic Deligne cohomology and, in
analogy with the Beilinson conjectures, conjectured that the image
is a full lattice. Finally, in some cases,  we made a conjecture on
the special value of the $L$-function.

One such case is that of the L-function of a surface at the integer
$s=1$. It is a generalization of the Tate conjecture for a variety
over a function field. The precise statement of this conjecture is
as follows

\begin{conj}
\label{conj} Let $X$ be a smooth proper surface  over a function
field $K$ and $\XX$ a semi-stable model of $X$ over $A$, its ring of
integers. Let $\Lambda(H^2(\bar{X},\Q_\ell),s)$ be the completed
$L$-function of $H^2$, namely the product of the local $L$-factors
at all places of $K$, where $\bar{X}=X \times Spec(\CI)$. Then,
there is a `thickened' regulator map $R_{\D}=\bigoplus_{\nu}
r_{\D,\nu} \oplus cl$
$$R_{\D}:H^3_{\M}(X,\Q(2))\oplus B^1(X) \longrightarrow \bigoplus_{\nu}
PCH^1(X_{\nu})$$
where $B^1(X)=CH^1(X)/CH^1_{hom}$ and $PCH^1(X_{\nu})$ is a subgroup
of the Chow group of the special fibre at $v$, which provides an
integral structure on the $\nu$-adic Deligne cohomology
$H^3_{\D}(X_{/\nu},\Q(2))$  defined below. We conjecture  $R_\D$
satisfies the following properties:
\begin{itemize}
\item[A.] $R_{\D}$ is a pseudo-isomorphism - namely it has a
finite kernel and co-kernel.
\item[B.](Tate's conjecture)~ $-\ord_{s=2} \Lambda(H^2(\bar{X},\Q_\ell),s)=\dim_{\Q}
B^1(X)\otimes \Q$
\item[C.]
$$\Lambda^*(H^2(\bar{X},\Q_\ell),1)=\pm \frac{|\text{coker}(R_{\D})|}{|\text{ker}(R_{\D})|}
\cdot \log(q)^{\ord_{s=1} \Lambda(H^2(\bar{X}),s)}$$
\end{itemize}
where $\Lambda^*$ denotes the first non-zero value in the Laurent
expansion and $|\;|$ of a finite set denotes its cardinality.
\end{conj}\vspace{.05in}

In other words,  the conjecture asserts that the regulator map
provides an isomorphism of the rational motivic cohomology with the
sum of all the $\nu$-adic Deligne cohomology groups. The special
value then measures the obstruction to this map being an isomorphism
of integral structures. 

This conjecture comes from the localization sequence for motivic cohomology which relates the motivic cohomologies of $X, \XX$ and $\XX_{\nu}$. The regulator map is the boundary map in the localization sequence.  We stated  the conjecture for surfaces - for points, when $X=Spec(K)$, this is simply a combination of the function field class number formula and units theorem -- the special value conjecture in this case implies the well known formula 
$$ \Lambda^* (H^0(Spec(X),0))=-\frac{h_K}{(q-1)\log(q)}$$
where $h_K$ is the class number and $(q-1)$ is the number of roots of  unity which can be interpreted as the orders of the kernel and cokernel of the regulator map respectively and the power of $\log(q)$ that appears corresponds to the well known fact that the zeta function has a simple pole at $s=1$.

Beilinson \cite{beil} theorem follows from a formula relating the
cohomological $L$-function of $ h^1(M_f) \otimes h^1(M_g)$, where
$h^1(M_f)$ and $h^1(M_g)$ are the motives of eigenforms of weight
two and some level $N$, to the regulator of an element of a certain
motivic cohomology group evaluated on the $(1,1)$-form
$\omega_{f,g}=f(z_1)\overline{g(z_2))}(dz_1\otimes
d\bar{z}_2-d\bar{z}_1 \otimes dz_2)$. We show an analogous formula
in the Drinfeld modular case with the Archimedean place being
replaced by the prime $\infty$. More precisely, since our
$L$-functions essentially take rational values, we have an exact
formula for the value analogous to the main theorem of \cite{basr}.

Our main result is the following --
\begin{thm}
\label{mainthm} Let $I$ be  a square-free element of $\F_{q}[T]$ and
$\Gamma_0(I)$ the congruence subgroup of level $I$.  Let $f$ and $g$
be Hecke eigenforms for  $\Gamma_0(I)$  and $\Lambda(h^1(M_f)\otimes
h^1(M_g),s)$ denote the completed, that is, with the $L$-factor at
$\infty$ included, $L$-function of the motive $h^1(M_f)\otimes
h^1(M_g)$.Then one has
\begin{equation}
\Lambda(h^1(M_f)\otimes h^1(M_g),1)=\frac{q}{ 2 (q-1)  \kappa}
(r_{\D,\infty} (\Xi_0(I)),\Z_{f,g})
\end{equation}
where $\Xi_0(I)$ is an element of motivic cohomology group
$H^3_{\M}(X_0(I) \times X_0(I),\Q(2))$, $r_{\D,\infty}$ is the
$\infty$-adic regulator map, $\kappa$ is an explicit integer
constant and $\Z_{f,g}$ is a special cycle in the special fibre at
$\infty$ and $(,)$ denotes the intersection pairing the Chow group
of the special fibre.
\end{thm}

\subsection{Outline of the paper}

In the first few sections we introduce some of the background on
Drinfeld modular curves. This is perhaps well known to people
working with function fields, but perhaps not so well known to
people working in the area of  algebraic cycles, hence it has been
included.

We then study the analytic side of the problem, namely the special
value of the $L$-function. We use the Drinfeld uniformization and an
analogue of the Rankin-Selberg method to get an integral formula for
the $L$-function. We also  formulate and prove an analogue of
Kronecker's first limit formula and use it to get an integral
formula for the special value at $1$ of the $L$-function.

Following  that  we  study  the  algebraic side  of  the  problem.  We
introduce the motivic cohomology group  of interest to us and define a
regulator  map on  it. This  regulator map  is the  boundary map  in a
localization sequence  relating the  motivic cohomology groups  of the
generic fibre and  special fibre. The result is  that the regulator of
an element of  our motivic cohomology group is  a certain $1$-cycle on
the special fibre.

We then construct an explicit element in this motivic cohomology
group using  analogues  of  the  classical  modular units  and
compute  its regulator.   The regulator  of this  element  is then
related to  our integral  formula  using  the   relation  between
components  of  the associated reduction of the Drinfeld modular
curve and vertices on the Bruhat-Tits tree.

In the classical case the regulator is a current on $(1,1)$-forms
and one obtains the special value by evaluating this current on a
specific form. Here, the regulator is a $1$-cycle and one obtains
the special value by computing the intersection pairing with a
specific cycle supported on the special fibre.  Finally we relate
our formula with the conjecture made above.

Curiously, the formulae are almost identical to the number field
case, though the objects involved are quite different. It suggests,
however, that there should be some underlying structure on which all
these results case be proved and the case of number field and
function fields arise by specializing to the case of $\ZZ$ or
$\F_q[T]$.

{\bf Acknowledgements:} I would like to thank S. Bloch, C.Consani,
J. Korman, S. Kondo,  M. Papikian, A Prasad and M. Sundara for their
help and comments on earlier versions of this manuscript. I would also like to thank 
the referee for his comments. 

I would  like to thank the University of Toronto, Max-Planck-Institute, Bonn 
and the TIFR Centre for Applicable Mathematics in Bangalore  for proving me an
excellent atmosphere in which to work in. Finally I would like to dedicate this paper 
to the memory of my mother, Ratna Sreekantan.

\newpage

\section{Notation}

Throughout this paper we use the following notation
\begin{itemize}
\item $\F_q$:  the finite field with $q=p^{n}$ elements, where $p$ is a prime number.

\item $A=\F_{q}[T]$: the polynomial ring in one variable.

\item $K=\F_q(T)$: the quotient field of $A$.

\item $\pi_{\infty}=T^{-1}$: a uniformizer at the infinite place
$\infty$.

\item $K_{\infty} = \F_q((\pi_{\infty}))$: the completion of $K$ at
$\infty$.

\item $K_{\infty}^{sep}$: the separable closure of
$K_{\infty}$.

\item $K_{\infty}^{ur}$: the maximal unramified extension of
$K_{\infty}$.

\item $\CI$:  the completed algebraic closure of $K_{\infty}$.

\item $\ord_{\infty} = - \deg$:  the negative value of the
usual degree function.

\item $\OO_{\infty} = \F_q[[\pi_{\infty}]]$:  the $\infty$-adic integers.

\item $|\cdot|$: the $\infty$-adic absolute value on $\KI$,
extended to $\CI$.

\item $|\cdot|_{i}$: the `imaginary part' of $|\cdot|$:
$|z|_{i}=\text{inf}_{x \in K_{\infty}}\{|z-x|\}$

\item $G$: the group scheme $GL_2$.

\item $B$: the Borel subgroup of $G$.

\item $Z$: the center of $G$.

\item $\K = G(\OO_{\infty})$.

\item $\I=\left\{ \begin{pmatrix} a & b\\c & d \end{pmatrix} \in \K \text{ such that } c \equiv
0 \text{ mod } \infty \right\}$

\item $\T$:  the Bruhat-Tits tree of $PGL_2(K_{\infty})$.

\item $V({\mathfrak G})$: the set of vertices  of a graph ${\mathfrak G}$.

\item $Y({\mathfrak G})$:  the set of oriented edges of an oriented graph ${\mathfrak G}$: if $e$ is
an edge, $o(e)$ and $t(e)$ denote the origin and terminus of the
edge.

\item $\m$: a divisor of $K$ with degree $\deg(\m)$
(this is different from $\deg(m)=-\ord_{\infty}(m)$ for $m \in K$).

\end{itemize}

\section{Preliminaries on Drinfeld modular curves}

In the function field setting, there are two analogues of the
complex  upper half-plane: the Bruhat-Tits tree and the Drinfeld
upper half-plane. These sets capture different aspects of the
classical upper half-plane. The Bruhat-Tits tree has a transitive
group action, but does not have a manifold structure, whereas the
Drinfeld  upper half-plane has the structure of a rigid analytic
manifold, but no transitive group action. These two sets are related
by means of the building map. We first describe the Bruhat-Tits
tree. We refer to the paper \cite{gere} for further details.

\subsection{ The Bruhat-Tits tree}

The Bruhat-Tits tree $\T$ of $PGL_2(K_{\infty})$ is an oriented
graph. It has the following description.

\subsubsection {Vertices and Ends of $\T$.}

The vertices of $\T$ consist of similarity classes $[L]$, where $L$
is a $\OO_\infty$-lattice in $(K_{\infty})^2$.  Recall that a
lattice $L$ is said  to be similar to  $L'$ ($L \equiv  L'$) if and
only  if there exists an element $c \in K_{\infty}^*$ such that
$L=cL'$. Two vertices $[L]$ and  $[L']$ are  joined by  an edge if
they are  represented by lattices    $L$    and    $L'$     with $L
\subset    L'$    and $\dim_{\F_{q}}(L'/L)=1$. Each vertex $v$ has
exactly $(q+1)$-adjacent vertices  and   this  set  is in bijection
with  $\CP^1(\F_{q})$.  More generally, the set  of vertices of $\T$
which are  adjacent to a fixed vertex   $[L]$    by at   most $k$
edges is  in bijection with  $\CP^1(L/\pi_{\infty}^kL)$.  This makes
$\T$ in to a $(q+1)$-regular tree.

A half-line is an infinite sequence of adjacent non-repeating
vertices $\{v_i\}$ starting with an initial vertex $v_0$. Two
half-lines are said to be equivalent if  the symmetric difference of
the two sets of vertices is a finite set. An end is an equivalence
class of half lines.

Let $\partial \T$ be the set of the ends of $\T$. There is a
bijection (independent of $L$)
\[
\partial   \T   \stackrel{\simeq}   {\longrightarrow}   \varprojlim_{k}
\CP^1(L/\pi_{\infty}^kL)       \simeq       \CP^1(\OO_{\infty})      =
\CP^1(K_{\infty}).
\]
The left-action of $G(K_{\infty})$ on $\T$ extends to an action on
$\partial \T$ which agrees with the action of $G(K_{\infty})$ on
$\CP(K_{\infty})$ by fractional linear transformations.

\subsubsection {Orbit Spaces.}

For $i\in\mathbb Z$, let $v_i \in V(\T)$ be the vertex
$[\pi_{\infty}^{-i}\OO_{\infty} \oplus \OO_{\infty}]$. As the vertex
$v_0$ has stabilizer $\K\cdot Z(\KI)$ in $G(\KI)$, one obtains the
following identification
\[
G(K_{\infty})/\K\cdot Z(\KI) \stackrel{\sim}{\rightarrow}
V(\T)\qquad g \mapsto g(v_0).
\]
Similarly, let $e_i$ be the edge $\overrightarrow{v_{i}v_{i+1}}$
({\it i.e}~$o(e_i) = v_i$, $t(e_i) = v_{i+1}$) then
\[
G(K_{\infty})/\I\cdot Z(\KI) \stackrel{\sim}{\rightarrow}
Y(\T)\qquad g \mapsto g(e_0).
\]

These identifications  allow one to consider functions  on vertices
and on edges of $\T$ as equivariant functions on matrices.

Let $w = \begin{pmatrix} 0&1\\1&0
\end{pmatrix}$. We set
\[
S_{V}=\left \{\begin{pmatrix}\pi_{\infty}^k & u \\0  & 1
\end{pmatrix}|~k \in \ZZ,~u \in \KI,~u~\text{mod}
\;\;\pi_{\infty}^k\OO_{\infty}\right \}
\]
and
\[
S_U=\left \{w\begin{pmatrix}1 &0 \\c & 1\end{pmatrix} |~c \in
\F_{q}\right \} \cup \{1\},\quad S_Y=\left \{gh~|~g  \in  S_V,  h
\in S_U \right \}.
\]
Then, $S_V$ is a system of representatives for $V(\T)$ and $S_Y$ is
a system of representatives for $Y(\T)$ \cite{papi}. We will use
these systems to define functions on the vertices and the edges of
the tree.

\subsubsection{Orientation.}\label{orien}

The choice of an end  $\infty$  representing the equivalence class
of the half line  $\{v_0,v_1,\ldots\}$, where $v_i$ are as above,
defines an orientation on $\T$ in the following manner. If $e=w_0w_1
$ is an edge, $e$ is said to be positively oriented if there is a
half line in the equivalence class of $\infty$ starting with initial
vertex $w_0$ and subsequent vertex $w_1$  and negatively oriented if
the half line has initial vertex $w_1$ and subsequent vertex $w_0$.
For a positively oriented edge, $e=w_0w_1$, let $o(e)=w_0$ denote
the origin of $e$ and $t(e)=w_1$ denote the terminus. This
determines a decomposition $Y(\T)=Y(\T)^{+}\cup Y(\T)^{-}$. We say
that $\sgn(e)=+1$ if $e \in Y(\T)^+$ and $\sgn(e)=-1$ if $e \in
Y^{-}(\T)$.

At a vertex $v$ there is precisely one positively oriented edge with
origin $v$  and there are $q$ positively oriented edges with
terminus $v$. That determines a bijection of $S_V$ with the set of
positively oriented edges $Y(\T)^+$. We will use the notation
$v(k,u)$ and $e(k,u)$ to denote the vertex and the positively
oriented edge represented by the matrix
$\begin{pmatrix} \pi_{\infty}^k & u \\ 0  & 1
\end{pmatrix}$
respectively. The edge  $e(k,u)$ has  origin $o(e)=v(k,u)$  and
terminus $t(e)=v(k-1,u)$.

\subsubsection{Realizations and norms.}

The realization  $\T(\R)$  of  the unoriented tree $\T$ is a
topological space consisting of  a real unit interval for every
unoriented edge of $\T$, glued together at the end  points according
to the incidence relations on $\T$.  If  $e$ is an edge, we denote
by $e(\R)$ the corresponding interval on the realization. Let
$\T(\ZZ)$ denote the points on $\T(\R)$ corresponding to the
vertices of $\T$. The set of points $\{t[L]+(1-t)[L']~|~t \in \Q\}$
lying on  edges $([L],[L'])$ will be denoted by $\T(\Q)$.

A norm on  a  $\KI$-vector space  $W$  is  a function $\nu:W
\rightarrow \R$ satisfying the following properties
\begin{itemize}
\item[{-}] $\nu(v) \geq 0;~\nu(v)=0 \Leftrightarrow v=0$

\item[{-}] $\nu(xv)=|x|\nu(v),~\forall~x \in \KI$

\item[{-}] $\nu(v+w)\leq \text{max}\{\nu(v),\nu(w)\},~\forall~v,w
\in W$.
\end{itemize}

Two norms $\nu_1$ and $\nu_2$  are said to  be similar if there
exist non-zero real constants $c_1$ and $c_2$ such that
$$c_1 v_1 \leq v_2 \leq c_2 v_1$$
The right action of $GL(W)$ on $W$ induces an action on the set of
norms as
\[
\gamma(\nu)(v)=\nu(v\gamma).
\]
This action descends to  similarity classes. The following theorem
relates norms to the realization of the tree.

\begin{thm}[Goldman-Iwahori]
There is  a canonical  $G(\KI)$-equivariant bijection $b$ between
the set $\T(\R)$ and the set of similarity classes of norms on
$W=\KI^2$.
\end{thm}
This bijection is  defined as follows.  To  a  vertex $[L]$ in
$\T(\ZZ)=V(\T)$ we associate $b([L])$, the class of the norm
$\nu_{L}$ defined by
\[
\nu_{L}(v) = \text{inf}\{|x|~:~x \in \KI, v \in xL\}.
\]
This norm makes  $L$ a unit  ball. If $P$  is a point of $\T(\R)$
which lies on the  edge $([L],[L'])$ with $\pi_{\infty} L' \subset L
\subset L'$ and $P=(1-t)[L]+t[L']$, then $b(P)$ is the class of the
norm defined by
\[
\nu_P(v) = \text{sup}\{\nu_{L}(v),q^t\nu_{L'}(v)\}.
\]

\subsection{Drinfeld's upper half-plane and the building map}

The set  $\Omega=\CP^1(\CI) -  \CP^1(\KI)=\CI -  \KI$ is called the
Drinfeld  upper half-plane. This space has  the structure of a rigid
analytic space over $\KI$. There is a canonical $G(\KI)$-equivariant
map
\begin{equation}\label{bm}
\lambda:\Omega \longrightarrow \T(\R)
\end{equation}
called the building map. It is defined as follows.  To $z \in
\Omega$, we associate  the similarity  class of the  norm $\nu_{z}$
on $(\KI)^2$ defined by
\[
\nu_{z}((u,v)) = |uz+v|.
\]
Since $|\cdot|$  takes values  in $q^{\Q}$, the  image of $\lambda$
in contained in $\T(\Q)$ and in fact one shows that
$\lambda(\Omega)=\T(\Q)$.

\subsection{The pure covering and its associated reduction}
\label{purecover} The Drinfeld upper half space is a rigid analytic
space. We need to study its reduction at the prime $\infty$.
However, there is no canonical reduction but there is a natural one
obtained by the associated analytic reduction of a certain pure
cover of $\Omega$. This is described in detail in \cite{gere}, pg
33-34 and in fact, we essentially copy from there.

The pure cover is described as follows. For $n\in \ZZ$, let $D_n$
denote the subset of $\C_{\infty}$ defined by
\begin{itemize}
\item $1.\; D_{n}=\{z \in \C_{\infty}:\;|\pi_{\infty}|^{n+1} \leq |z| \leq
|\pi_{\infty}|^{n}\}$
\item $2.\; |z-c\pi_{\infty}^n|\geq |\pi_{\infty}|^n,
|z-c\pi_{\infty}^{n+1}| \geq |\pi_{\infty}^{n+1}|$ for all $c\in
\F_{q}^{*} \subset K_{\infty}.$
\item Equivalently $2'.\; |z|=|z|_i$
\end{itemize}
Condition $2'$ shows $D_{n} \subset \Omega$ and is independent of
the choice of $\pi_{\infty}$. This is an affinoid space over
$K_{\infty}$ with ring of holomorphic functions
$$A_n=K_{\infty}<\pi_{\infty}^{-n}z,\pi_{\infty}^{n+1}z^{-1},(\pi_{\infty}^{-n}z-c)^{-1},(\pi_{\infty}^{-(n+1)}z-c)^{-1}|c
\in \F_{q}^{*}>$$
which is the algebra of `strictly convergent power series' in
$\pi_{\infty}^{-n}z$. This allows one to define the canonical
reduction $(D_n)_{\infty}$ and this is isomorphic to the union of
two projective lines meeting at an $\F_{q}$-rational point, with all
other rational points deleted.

For $\bi=(n,x),n\in \ZZ,x \in K_{\infty}$ let
$D_{\bi}=D_{(n,x)}=x+D_n$. Then one can see that, if $\bi'=(n',x')$,
$$D_{\bi}=D_{\bi}' \Leftrightarrow n=n' \text{ and } |x-x'|\leq
|\pi_{\infty}|^{n+1}$$
So if $I=\{(n,x)|n\in \ZZ, x\in K_{\infty}/\pi_{\infty}^{n+1}
\OO_{\infty}\}$, where, for each $n$, $x$ runs through a set of
representatives, then
$$\Omega=\bigcup_{\bi \in I} D_{\bi}$$
is a pure covering of $\Omega$. For any $\bi$, there are only
finitely many $\bi'$ such that $D_{\bi} \cap D_{\bi'}\neq \phi$.

With respect to this covering one has an associated analytic
reduction,
$$R:\Omega \longrightarrow \Omega_{\infty}$$
where $\Omega_{\infty}$ consists of a union of $\CP^1_{\F_{q}}$'s
each of which meets $q+1$ other ones at $\F_q$ rational points.
Conversely, any $\F_q$ rational point $s$ of a component $M$
determines a component $M'$ such that $M'\cap M=\{s\}$. For adjacent
$M$ and $M'$ let $M^*=M-M(\F_q)$ and $(M \cup M')^*=M\cup
M'-(M(\F_q) \cup M'(\F_q)) \cup (M \cap M')$. Then, there exits
$\bi,\bi'$ such that
$$R^{-1}(M^*)=D_{\bi} \cap D_{\bi'} \text{ and } R^{-1}((M \cup
M')^*)=D_{\bi}.$$

The intersection graph of $\Omega_{\infty}$ is the graph whose
vertices are the components $M$ of $\Omega_{\infty}$. Two vertices
$M$ and $M'$ are joined by an oriented edge if and only if $M$ and
$M'$ are adjacent components of $\Omega_{\infty}$, that is, if
$M\cap M'\neq \phi$. The map $\lambda$ in \eqref{bm} determines a
canonical identification of this graph with the Bruhat-Tits tree:
Given a component $M$ there exists a unique $[L] \in \T(\ZZ)$ such
that
$$\lambda^{-1}([L])=R^{-1}(M^{*})$$
and this association is compatible with the group actions and
identifies the two graphs. We will use this identification rather
crucially  in the final step of the proof.

\subsection{Drinfeld modular curves of level $I$}

For a monic polynomial $I$ in $A$ let
$$
\Gamma_0(I)=\left\{\begin{pmatrix} a & b \\
c & d \end{pmatrix}\in \Gamma~|~c \equiv 0 \;\; \text{mod}
\;\;I\right \}.
$$
$\Gamma_0(I)$ acts discretely on $\Omega$ via M\"obius
transformations: For $z \in \Omega$ and $\gamma=\begin{pmatrix} a& b
\\c & d \end{pmatrix} \in \Gamma_0(I)$, define
$$\gamma z=\frac{az+b}{cz+d}.$$

The Drinfeld modular curve $X_0(I)$ of  level $I$ is a smooth,
proper, irreducible  algebraic curve,  defined over $K$, such that
its $\CI$ points have the structure of a rigid analytic space and
there is a canonical isomorphism of analytic spaces over $\CI$
$$
X_0(I)(\CI) \simeq \Gamma_0(I)\backslash \Omega\cup
\{\text{cusps}\}.
$$
where the cusps are finitely many points in bijection with
$\Gamma_0(I)\backslash \CP^1(K)$. Entirely analogous to the
classical construction over a number field, the Drinfeld modular
curve $X_0(I)$ parameterizes Drinfeld modules of rank two with level
$I$ structure.

Let
$$\T_0(I)=\Gamma_0(I)\backslash \T$$
denote the corresponding quotient of the Bruhat-Tits tree by the
left action of $\Gamma_0(I)$. Let $X(\T_0(I))$ and $Y(\T_0(I))$
denote the vertices and edges of the graph $\T_0(I)$ respectively.
$\T_0(I)$ is an infinite graph consisting of a finite graph
$\T_0(I)^0$ and a finite number of ends corresponding to the
finitely many cusps (\cite{gere},Section 2.6).

The curve $X_0(I)$ is a totally split curve over $K_{\infty}$. The
pure covering of $\Omega$ induces a pure covering of $X_0(I)$ and
the associated analytic reduction $R$ is a scheme $X_0(I)_{\infty}$
over $\F_q$ which is a finite union of $\CP^1_{\F_q}$'s intersecting
at $\F_q$ rational points. The intersection graph of this scheme is
the finite part $\T_0(I)^0$ of the graph $\T_0(I)$.

\subsection{Harmonic cochains on the Bruhat-Tits tree}

In the function field setting there are two notions of modular
forms -- corresponding to the two analogues of the complex upper half
plane. One notion deals with  certain equivariant functions on the
Drinfeld upper half plane while the other  refers to  certain
invariant harmonic co-chains on the Bruhat-Tits' tree. The latter
are sometimes called automorphic forms of Jacquet-Langlands-Drinfeld
(JLD) type. It is these that we will be concerned with and will
review their definition and properties in this section.

If $R$ is a commutative ring, an $R$-valued harmonic co-chain on
$Y(\T)$ is a map $\phi:Y(\T) \longrightarrow R$ satisfying the
harmonic conditions:
\begin{itemize}
\item[{-}] $\phi(e) + \phi(\overline{e})=0$

\item[{-}] $\displaystyle{\sum_{t(e)=v}} \phi(e)=0$
\end{itemize}
where, for $e$ in $Y(\T)$,  $\overline{e}$ denotes the same edge
with the opposite orientation. The second condition can also be
stated as follows - First, notice that there is precisely one edge
$e_0$  with $t(e_0)=v$ and $sgn(e_0)=-1$. The second condition is
then equivalent to
\[
\phi(e_0)=\sum_{t(e)=v\atop sgn(e)=1} \phi(e).
\]
If $\Gamma$ is a subgroup of $G(A)$ we will consider co-chains
satisfying the further condition of $\Gamma$-invariance, namely,
\begin{itemize}
\item[{-}] $\phi(\gamma e)=\phi(e),\quad \forall \gamma \in
\Gamma.$
\end{itemize}
The   group  of  $\Gamma$-invariant,   $R$-valued harmonic co-chains
on the edges of $\T$ is  denoted by $H(Y(\T),R)^{\Gamma}$. The
harmonic functions on  the edges of $\T$ are the analogues of
classical  cusp forms of weight  $2$. In fact, if $\ell  \neq p$  is
a  prime number,  the $\Gamma$-invariant harmonic co-chains  detect
`half'  of  the \'etale cohomology group
$H^1_{\text{\'et}}(X(\Gamma),\Q_{\ell})$ (\cite{teit}, pg. 272).

An  R-valued harmonic cochain  $f$ is said to be of   level  $I$ if
 $f \in H(Y(\T),R)^{\Gamma_0(I)}$.   If  $f$ has finite support as
a function on $\Gamma_0(I)\backslash Y(\T)$, it is called a cusp
form or said to be cuspidal. Usually we will deal with $\ZZ$,$\C$ or
$\Q_{\ell}$ valued functions. In  analogy with the classical case,
we sometimes will use the word `form' to denote these functions.

\subsubsection{Fourier expansions.}

A  harmonic  function  on  the  set of positively oriented edges
$Y(\T)^+$ which is invariant under  the action of the group
\[
\Gamma_{\infty}=\left \{\begin{pmatrix} a & b \\0 & d \end{pmatrix}
\in G(A)\right \}
\]
has a Fourier expansion. This statement is a consequence of the
general theory of Fourier analysis on ad\`ele groups. Details can be
founds in \cite{geke1}. This expansion has the following
description. Let $\eta:\KI \rightarrow {\mathbb C}^*$ be the
character defined as
\[
\eta\left (\sum_j a_j \pi_{\infty}^j\right ) = \exp\left (\frac{2
\pi i {\mathrm Tr}(a_1)}{p}\right )
\]
where ${\mathrm Tr}$ is the trace map from  $\F_q$ to $\F_p$. Then,
the Fourier expansion of a $\Gamma_{\infty}$-invariant function $f$
on $Y(\T)^+$ is given by
\[
f\left (\begin{pmatrix} \pi_{\infty}^k & u \\0 & 1
\end{pmatrix}\right ) = c_0(f,\pi_{\infty}^k) + \sum_{0 \neq m \in A \atop \deg(m)
\leq k-2} c(f,\div(m)\cdot \infty^{k-2}) \eta(mu).
\]
The constant Fourier  coefficient $c_0(f,\pi_{\infty}^k)$ is the
function of $k \in \ZZ$ given by
\[
c_0(f,\pi_{\infty}^k) = \begin{cases} f\left (\begin{pmatrix}
\pi_{\infty}^k & 0
\\0 & 1
\end{pmatrix}\right) &  {\mathrm if} \; k\leq  1 \\ q^{1-k}  \displaystyle{\sum_{u \in
(\pi_{\infty})/(\pi_{\infty}^k)} f\left(\begin{pmatrix}
\pi_{\infty}^k & u \\0 & 1
\end{pmatrix}\right)} & {\mathrm if} \; k \geq 1. \end{cases}
\]
For a non-negative  divisor $\m$  on $K$,  with $\m=\div(m) \cdot
\infty^{\deg(\m)}$, the non-constant Fourier coefficient is
\[
c(f,\m)=q^{-1-\deg(\m)} \sum_{u    \in
(\pi_{\infty})/(\pi_{\infty}^{2+\deg(\m)})}
 f\left(\begin{pmatrix} \pi_{\infty}^{2+\deg(\m)}& u \\0 & 1  \end{pmatrix}\right )
\eta(-mu).
\]

\subsubsection{Petersson inner product.}
\label{petersson} There is  an analogue  of the Petersson  inner
product for invariant functions on the tree $\T$.

If $f$ and $g$ are complex valued harmonic co-chains for
$\Gamma_0(I)$, one of which is cuspidal, define
\[
\delta(f,g)(e) =f(e)\overline{g(e)}d\mu(e) \qquad \text { for } e\in
Y(\T_0(I))
\]
where $\mu(\cdot)$ is the Haar measure on the discrete set
$Y(\T_0(I))$ defined by $\mu(e) =
\frac{q-1}{2}|\text{Stab}_{\Gamma_0(I)}(e)|^{-1}$, where
$|\text{Stab}_{\Gamma_0(I)}(e)|$ is the cardinality of the
stabilizer of $e\in Y(\T_0(I))$. The Petersson inner product of $f$
and $g$ is defined as
\[
<f,g> =   \int_{Y(\T_0(I))}   \delta(f,g)= \int_{Y(\T_0(I))} f(e)
\overline{g(e)}d\mu(e).
\]

\subsubsection{Hecke operators and Hecke eigenforms.}
\label{heckeoperators}

Let $\PP$ be a prime  and $I$ a fixed level. The Hecke operator
$T_{\PP}$ is the operator on $H(Y(\T),\C)^{\Gamma_0(I)}$ defined by
\[
T_{\PP}(f)(e)=\begin{cases}  f\left(e \begin{pmatrix} \PP & 0 \\ 0
& 1
\end{pmatrix}  \right )  + \displaystyle{\sum_{r\text{ mod } \PP}} f\left (e \begin{pmatrix}  1 & r
\\ 0 & \PP \end{pmatrix}  \right ) & if\;\;\PP \not{|}~I\\
\displaystyle{\sum_{r\text{mod}~\PP}} f\left(e \begin{pmatrix}  1 & r \\
0 & \PP
\end{pmatrix}   \right) & if \;\;\PP~|~I. \end{cases}
\]

A Hecke  eigenform $f$ is a harmonic co-chain of level $I$ which is
an eigenfunction of all  the Hecke operators  $T_{\PP}$. $f$ is
called a newform if in addition it lies in the orthogonal
complement, with respect to the Petersson inner product, of the
space generated by all cusp forms of level $I'$ for all levels $I'$
properly dividing $I$.

If $f$ is a non-zero newform, then the coefficient $c(f,1)$ in the
Fourier expansion is not zero. The form $f$ is said to be normalized
if one further assumes that $c(f,1)=1$. Let $\lambda_{\PP}$ denote
the eigenvalue of the Hecke operator $T_{\PP}$. The Fourier
coefficients of a cuspidal, normalized newform $f$ have the
following special properties:
\begin{itemize}
\item[{-}] $c_0(f,\pi_{\infty}^k)=0,\quad \forall k \in \ZZ$

\item[{-}] $c(f,1)=1$

\item[{-}] $c(f,\m)c(f,{\mathfrak n}) = c(f,\m {\mathfrak n})$,
\quad whenever $\m$ and ${\mathfrak n}$ are relatively prime

\item[{-}] $c(f,\PP^{n-1})-\lambda_{\PP} c(f,\PP^n) +
|\PP|c(f,\PP^{n+1})=0$,\quad if $\PP \not{|}~I \cdot \infty$

\item[{-}] $c(f,\PP^{n+1})-\lambda_{\PP} c(f,\PP^n)=0$, \quad if
$\PP~|~ I$

\item[{-}] $c(f,\infty^{n-1})=q^{-n+1}$, \quad if $n \geq 1.$
\end{itemize}
If $f$  and $g$ are  normalized Hecke eigenforms  and $f \neq g$,
then $<f,g>~= 0$. Further, since the Hecke operators are self
adjoint, $f=\bar{f}$.

\subsubsection{Logarithms and the logarithmic derivative.}

Let $f$ be a $\C_{\infty}$-valued invertible function on $\Omega$. There is a notion of
the  logarithm of $|f|$ defined as follows. Let  $v$ be a vertex of
$\T$ and $\tau_{v} \in \Omega$ an element of $\lambda^{-1}(v)$
where $\lambda$ is the building map defined in \eqref{bm}. From Section \ref{purecover} one can see that
the function $|\cdot|$ factors through the building map, so the quantity $|f|$  depends only on $v$ and not on the choice of $\tau_v$.

Define
\begin{equation}\label{thelog}
\log|f|(v)=\log_{q} |f(\tau_v)|
\end{equation}
This function  takes values in $\ZZ$.

If $g$ is  a function on the vertices of the tree $\T$, then the
derivative of $g$ is a function on the edges of $\T$  defined to be
\begin{equation}
\label{log}
\partial g(e) = g(t(e))-g(o(e)).
\end{equation}
The logarithmic derivative of an invertible function $f$ on $\Omega$
is the composite of these two maps, namely
\begin{equation}
\label{dlog}
\partial \log|f| (e) = \log|f|(t(e))-\log|f|(o(e)).
\end{equation}

\subsubsection{The cohomology of a Drinfeld modular curve.}

The cohomology of a Drinfeld modular curve has a decomposition, due
to Drinfeld, which is analogous to the classical decomposition of
the cohomology of a modular curve into eigenspaces of modular forms
of weight two.

Let $\ell$ be a prime, $\ell \neq p$. There is a two dimensional $\ell$-adic representation $\SSP_\ell$ of
the Galois group $Gal(K_{\infty}^{sep}/K_{\infty})$, called the
special representation, which acts through a quotient isomorphic to
$\hat{\ZZ}\ltimes \ZZ_{\ell}(1)$. The group $\hat{\ZZ}$ is
isomorphic to $Gal(K_\infty^{ur}/K_\infty)$. The canonical generator
of $\hat\ZZ$ corresponds to $F_\infty$, the Frobenius automorphism
of $K_\infty^{ur}/K_\infty$. The group $\ZZ_\ell(1)$ is isomorphic
to $Gal(E_\ell/K_\infty^{ur})$, where $E_\ell/K_\infty^{ur}$ is the
field extension obtained by adjoining all the $\ell^r$-th roots of
the uniformizer $\pi_{\infty}$ to $K_\infty^{ur}$. The action of
$F_\infty = 1 \in \hat\ZZ$ on $\ZZ_\ell(1)$ is given by $F_\infty u
F_\infty^{-1} = u^q$, for $u\in\ZZ_\ell(1)$. Choose an isomorphism
$\ZZ_\ell(1) \cong \ZZ_\ell$, then
\begin{equation*}
\SSP_\ell: Gal(K_{\infty}^{sep}/K_\infty) \twoheadrightarrow \hat\ZZ
\ltimes \ZZ_\ell \to Gl(2,\mathbb Q_\ell)
\end{equation*}
where the right hand-side arrow is defined as
\begin{equation}
(1,0) = F_\infty \mapsto \begin{pmatrix}1 & 0 \\0 &
q^{-1}\end{pmatrix};\quad (0,1) \mapsto
\begin{pmatrix} 1 & 1 \\0 & 1\end{pmatrix}. \label{frob}
\end{equation}\vspace{.1in}
We recall the following theorem of Drinfeld \cite{gere}.
\begin{thm}[Drinfeld] Let $X_0(I)$ be a Drinfeld modular curve
with level $I$ structure and let $\bar{X}_0(I)=X_0(I) \times
Spec(\CI)$. Then
\begin{equation}
H^1_{\text{\'et}}(\bar{X}_0(I),\Q_{\ell}) \cong
H(Y(\T),\Q_{\ell})^{\Gamma_0(I)} \otimes \SSP_\ell. \label{drinisom}
\end{equation}
This isomorphism is compatible with the action of the local Galois
group $Gal(K_{\infty}^{sep}/K_{\infty})$ and the action of the Hecke
operators.
\end{thm}

A consequence of this theorem and Eichler-Shimura type relations
\cite{gere}(4.13.2) is the decomposition of the $L$-function of a
Drinfeld modular curve into a product of $L$-functions of Hecke
eigenforms
\begin{equation}
L(H^1_{\text{\'et}}(\bar{X}_0(I)),s)=\prod L(h^1(M_f),s),\qquad
s\in\C. \label{dec0}
\end{equation}
where $L(h^1(M_f),s)=L(f,s)$ is the $L$-function of the motive
$h^1(M_f)$ corresponding to the Hecke eigenform $f$ (\cite{papi}, pg
332).

In this paper, we focus on the study of the $L$-function of
$H^2(\bar{X}_0(I) \times \bar{X}_0(I),\mathbb Q_\ell)$. Applying the
K\"unneth formula we get the following decomposition
\begin{equation}\label{dec}
L(H^2_{\text{\'et}}(\bar{X}_0(I)\times
\bar{X}_0(I)),s)=L(H^2_{\text{\'et}}(\bar{X}_0(I)),s)^2
L(H^1_{\text{\'et}}(\bar{X}_0(I))\otimes
H^1_{\text{\'et}}(\bar{X}_0(I)),s).
\end{equation}

The incomplete ( here we omit the local factor at $\infty$ )
$L$-function of $H^2_{\text{\'et}}(\bar{X}_0(I))$ is
$\zeta_A(s-1)=\frac{1}{1-q^{2-s}}$. Under the assumption that the
level is square-free and using the decomposition above \eqref{dec0}
the $L$-function of the last factor in \eqref{dec} can be expressed
as a product
$$L(H^1(\bar{X}_0(I))\otimes H^1(\bar{X}_0(I)),s) = \zeta_A(2s)^{-1} \prod_{f,g} L_{f,g}(s)$$
where $f$ and $g$ are normalized newforms of JLD type and level $I$.
$L_{f,g}(s)$ is the Rankin-Selberg convolution $L$-function defined
in the next section. It is essentially the $L$-function of the
tensor product of the motives $h^1(M_f) \otimes h^1(M_g)$.

\section{The Rankin-Selberg convolution}
\label{rankin}

The main goal of this section is the computation of a special value
of the  convolution $L$-function  of two automorphic forms of
JLD-type verifying  certain prescribed conditions.

We begin by studying certain Eisenstein series  on the Bruhat-Tits
tree $\T$. The classical Eisenstein-Kronecker-Lerch series are real
analytic functions on  the upper half-plane, invariant under the
action of a congruence subgroup $\Gamma\subset SL_2(\ZZ)$. They are
related to logarithms of modular units on the associated modular
curve via the Kronecker Limit formulas.

There are function field analogues of these series, as well as an
analogue of Kronecker's First Limit formula. These results follow
from the work of Gekeler \cite{geke1} and they are the crucial steps
in the process of relating the regulators of elements  in $K$-theory
to special values of $L$-functions.

\subsection{ Eisenstein series}\label{Eis}

The real analytic Eisenstein series for $\Gamma_0(I)$ is defined as
\[
E_I(\tau,s)= \displaystyle{\sum_{\gamma \in
\Gamma_\infty\backslash\Gamma_0(I)}} |\gamma(\tau)|_{i}^{s},\qquad
\tau \in \Omega,~s\in\C.
\]
This series converges absolutely for $Re(s) \gg 0$ and from the
definition one can see that it is $\Gamma_0(I)$ invariant.

The  `imaginary part' function $|z|_{i}=\text{inf}\{|z-x|;~x \in
K_{\infty}\}$ factors  through the building map, so the Eisenstein
series can be thought of as a function defined on the vertices of
the Bruhat-Tits tree. In terms of the matrix representatives
$S_{V}$, $E_I(\tau,s)$ can be expressed as follows.

Let $m,n \in A,~(m,n)\neq (0,0)$ and let $v \in V(\T)$ be a vertex
represented  by  $v=\begin{pmatrix}  \pi^k  &   u\\0  &  1
\end{pmatrix}$. For $\omega=\ord_{\infty}(mu+n)$ and $s\in\C$, define
\begin{equation}\label{thephi}
\phi_{m,n}^{s}(v)=\phi^s_{m,n} \begin{pmatrix} \pi^k  & u  \\ 0 & 1
\end{pmatrix} =
\begin{cases}  q^{(k-2\deg(m))s}   &  if\;\;\  \omega   \geq  k-\deg(m)\\
q^{(2\omega-k)s} & if \;\; \omega < k-\deg(m). \end{cases}
\end{equation}
Then, using an explicit set of representatives for
$\Gamma_\infty\backslash\Gamma_0(I)$, we have ( see \cite{papi},
section~4 for details)
\begin{equation*}
\label{E-I} E_{I}(v,s)=q^{-ks}+\sum_{{m\in A\atop m~\text{ monic
}}\atop m\equiv 0~\text{mod }I}\sum_{n \in A,\atop (m,n)=1}
\phi^{s}_{m,n}(v).
\end{equation*}

Let $E(v,s)=E_{1}(v,s)$.  In \cite{papi},  it is shown that
$E_I(v,s)$ has an analytic continuation to a meromorphic function on
the entire complex plane, with a simple pole at $s=1$. Lemma 3.4 of
\cite{papi} relates the two series $E$ and $E_I$ through the formula
\begin{equation}\label{relat}
\zeta_I(2s)E_I(v,s)=\frac{\zeta(2s)}{|I|^s} \sum_{d|I\atop d\text{
monic }} \frac{\mu(d)}{|d|^s} E((I/d)v,s)
\end{equation}
where $\zeta(s)=\frac{1}{1-q^{1-s}}$ is the zeta function of $A$, $\zeta_I(s) =
\displaystyle{\prod_{\PP  \nmid I}} (1-|\PP|^{-s})^{-1}$,
$\mu(\cdot)$ is the M\"obius function of $A$, defined entirely
analogously to the usual M\"obius function using the monic prime
factorization of an element of $A$, and $(I/d)v$ denotes the action
of the matrix $\begin{pmatrix} I/d &0\\0 & 1
\end{pmatrix}$ on $v$.

Notice that a function  $F$ on the vertices of $\T$ can be
considered as a function on the  edges of the tree by defining
$F(e)=F(o(e))$. In particular, if we define
\begin{equation}\label{Eisfunct}
E_I(e,s)=E_I(o(e),s),\quad e\in Y(\T)
\end{equation}
we recover the definition given in section~3 of \cite{papi}.

\subsubsection{Functional equation.}

The  Eisenstein  series  $E(e,s)$  satisfies a  functional equation
analogous to that satisfied by the  classical
Eisenstein-Kronecker-Lerch series. The analogue of the `archimedean
factor' of the zeta function of $A$ is
\begin{equation}
L_{\infty}(s)=(1-|\infty|^{s})^{-1}=\frac{1}{1-q^{-s}}.
\label{lfactorinfty}
\end{equation}
We recall the following result
\begin{thm}
\label{funeqv} Define $\Lambda(e,s) = -L_{\infty}(s) E(e,s)$. Then,
$\Lambda(e,s)$ has a simple pole at $s=1$  with residue $-(\log q)^{-1}$ and satisfies the functional equation
\[
\Lambda(e,s)=-\Lambda(e,1-s).
\]
\end{thm}
\begin{proof} \cite{papi}~Theorem~3.3.
\end{proof}\vspace{.1in}

\subsection{The Rankin-Selberg convolution}
\label{phi}

In \cite{papi}, M.~Papikian describes a function field  analogue of
the Rankin-Selberg formula.  In this section we will apply this
result by using the interpretation of the Eisenstein series
$E_{I}(v,s)$ as an automorphic form on the edges of the Bruhat-Tits
tree $\T$.

Let $f$ and $g$ be two automorphic forms of level $I$ on $\T$.
Consider the series
\[
L_{f,g}(s) = \zeta_I(2s) \sum_{\m \; \text{effective divisors}\atop
(\m,\infty)=1} \frac{c(f,\m) \bar c(g,\m)}{|\m|^{s-1}}
\]
If $f$ and $g$ are normalized newforms, using the decomposition
$\m=\m_{fin}\infty^d$ ($d \geq 0$) and the last property of the
Fourier coefficients listed in section \ref{heckeoperators}, namely
$$c(f,\m)=c(f,\m_{fin}\cdot   \infty^d)=c(f,\m_{fin})q^{-d}$$
so we can pull out the Euler factor at $\infty$ and we have
\[
\zeta_I(2s)   \sum_{\m~\text{effective divisors}}   \frac{c(f,\m)
\bar c(g,\m)}{|\m|^{s-1}}=L_{\infty}(s+1) L_{f,g}(s).
\]

\begin{prop}[``Rankin's trick"]\label{RT}
Let $f$ and $g$ be two cusp forms of level $I$. Then
\[
\zeta_{I}(2s) <f~E_I(e,s),~g>~=~\zeta_I(2s) \int_{Y(\T_0(I))}
E_I(e,s)f(e)\overline{g(e)} d\mu(e) =
q^{1-2s}L_{\infty}(s+1)L_{f,g}(s)
\]
\end{prop}

\begin{proof} \cf~\cite{papi}, section 4.
\end{proof}

We set $\Phi(s) = \Phi_{f,g}(s)$ where
\begin{equation}\label{thefunct}
\Phi_{f,g}(s)
:=-\frac{\zeta_I(2s)L_{\infty}(s)|I|^{s}}{\zeta(2s)}\int_{Y(\T_0(I))}
E_I(e,s) f(e) \overline{ g(e)} d\mu(e).
\end{equation}
It follows from the proposition above, \eqref{relat},
\eqref{Eisfunct} and Theorem~\ref{funeqv} that $\Phi(s)$ has the
following description
\begin{align}\label{thefunct1}
\Phi(s) &= \sum_{d|I\atop d\text{ monic }} \frac{\mu(d)}{|d|^s}
\int_{Y(\T_0(I))} \Lambda ((I/d)e,s)f(e)\overline {g(e)}d\mu(e) \\
&=-q^{1-2s}|I|^sL_{f,g}(s)L_{\infty}(s)L_{\infty}(s+1)\zeta(2s)^{-1}.\notag
\end{align}
$\Phi(s-1)$ is the {\it completed}, namely with the factors at
$\infty$ included,  $L$-function of the motive $h^1(M_f) \otimes
h^1(M_g)$,
$$\Phi_{f,g}(s-1)=\Lambda(h^1(M_f)\otimes
h^1(M_g),s).$$

\vspace{.1in}

For future use we recall the following result
\begin{thm}[Rankin]\label{therankin}
$L_{f,g}(s)$ has a simple pole at $s=1$ with residue a non-zero
multiple of $<f,g>$, the Petersson inner product of $f$ and $g$. In
particular, if $f$ and $g$ are normalized newforms and $f \neq g$,
then $<f,g>=0$, so $L_{f,g}(s)$ does not have a pole at $s=1$.
\end{thm}
It follows that if $f$ and $g$ are normalized newforms and $f\neq
g$, then $\Phi(1) = \frac{q|I|L_{f,g}(1)}{(1-q^2)}$. To compute the
value $\Phi(0)$ we will make use of the following theorem.

\begin{thm}\label{fctequ} Let $f$ and $g$ be two cuspidal eigenforms of square-free levels
$I_1,I_2$ respectively, with $I_1$ and $I_2$ co-prime monic
polynomials. Let $I = I_1I_2$. Then, the function $\Phi(s)$ as
defined in \eqref{thefunct} satisfies the functional equation
\begin{equation*}
\Phi(s) = -\Phi(1-s).
\end{equation*}

\end{thm}

\begin{proof}
The method of the proof  is similar to that of Ogg \cite{ogg},
section 4.  Using the Atkin-Lehner operators, we can simplify the
integral in \eqref{thefunct}. Let $\PP$ be a monic prime element  of
$A$ such that $\PP|I$, say $\PP|I_1$. The Atkin-Lehner operator
$W_{\PP}$ corresponding to $\PP$ is represented by
\[
\beta=\begin{pmatrix} a\PP & -b \\I & \PP \end{pmatrix}, \;
\det(\beta) = \PP,\quad \beta \in \Gamma_0(I/\PP) \begin{pmatrix}
\PP & 0 \\0 & 1 \end{pmatrix},\qquad \text{for some}~a,b\in A.
\]
Let $d$ be a monic divisor of $I/\PP$, then
\[
\begin{pmatrix} I / \PP d & 0 \\ 0 & 1
\end{pmatrix} \beta \begin{pmatrix} I/d & 0 \\ 0 & 1
\end{pmatrix}^{-1} \in \Gamma.
\]
As $\Lambda(e,s)$ of Theorem~\ref{funeqv} is $\Gamma$-invariant we
have
\begin{equation*}
\Lambda((I/\PP d)\beta e,s)=\Lambda((I/d)e,s)
\end{equation*}
where $(I/\PP d) = \begin{pmatrix} I / \PP d & 0 \\ 0 & 1
\end{pmatrix}$.  Since  $\beta$ normalizes $\Gamma_0(I_1)$ and $f$ is a
newform, we obtain
\[
f|_{\beta}=f_{W_{\PP}}=c(f,\PP)f
\]
where $c(f,\PP)=\pm 1$. Further,  if $h=g|_{\beta}$, then
$h|_{\beta}=g|_{\beta^2}=g$. We then have
\begin{align}\label{int}
\int_{Y(\T_0(I))} \Lambda((I/\PP d)e,s)\delta(f,g) &=
\int_{\beta^{-1}(Y(\T_0(I)))}
\Lambda((I/d)e,s)c(f,\PP)\delta(f,g|_{\beta}) \\
&= \int_{Y(\T_0(I))} \Lambda((I/d) e,s) c(f,\PP)\delta(f,h)\nonumber
\end{align}
as $\beta^{-1}(Y(\T_0(I)))$ is a fundamental domain for $\beta^{-1}
\Gamma_0(I) \beta=\Gamma_0(I)$. We deduce that
\begin{equation}
\label{fe1} \Phi(s)= \sum_{d|(I/\PP)\atop d\text{ monic
}}\frac{\mu(d)}{|d|^s} \int_{Y(\T_0(I))} \Lambda((I/d)e,s)(
\delta(f,g)+c(f,\PP)\delta(f,h)).
\end{equation}
Now, we repeat this process with $h$ in the place of $g$. It follows
from the Fourier expansion that
\[
L_{f,h}(s)=c(f,\PP)|\PP|^{-s}L_{f,g}(s).
\]
Substituting this expression in \eqref{fe1}, we have
\begin{align}\label{fe2}
c(f,\PP)|\PP|^{-s}\Phi(s) &= \sum_{d|(I/\PP)\atop d\text{ monic }}
\int_{Y(\T_0(I))}
(\Lambda((I/d)e,s)-|\PP|^{-s}\Lambda((I/\PP d) e,s) ) \delta(f,h) \\
&= \sum_{d|(I/\PP)\atop d\text{ monic }} \frac{\mu(d)}{|d|^{s}}
\int_{Y(\T_0(I))} \Lambda((I/d)e,s) \left(
\delta(f,h)+\frac{c(f,\PP)}{|\PP|^{s}}\delta(f,g) \right). \nonumber
\end{align}
We denote by ${\mathfrak S}$ the sum of the terms involving
$\delta(f,h)$. Comparing \eqref{fe1} and \eqref{fe2}, we obtain
$$ {\mathfrak S}(\frac{c(f,\PP)}{|\PP|^{s}}-1)=0.$$
The only way this equation  can hold for all $s$ in $\C$ is if
${\mathfrak S}=0$. Hence we obtain
\[
\Phi(s)=\sum_{d|(I/\PP)\atop d\text{ monic }} \frac{\mu(d)}{|d|^{s}}
\int_{Y(\T_0(I))} \Lambda((I/d)e,s) \delta(f,g).
\]
Repeating  this process for all primes $\PP$ dividing $I$, keeping
in mind the assumption that a prime divides $I_1$ or $I_2$ but not
both. We get
\begin{equation*} \Phi(s)=\int_{Y(\T_0(I))}
\Lambda((I)e,s)\delta(f,g).
\end{equation*}
As $\Lambda((I)e,s)$ satisfies the functional equation
$\Lambda((I)e,s)= -\Lambda((I)e,1-s)$, we finally obtain
\[
\Phi(s)=-\Phi(1-s).
\]
\end{proof}

\subsection{Kronecker's limit formula and the Delta function}

For the computation of the value $\Phi(0)$ we introduce the Drinfeld
discriminant function $\Delta$.

\subsubsection{The discriminant function and the Drinfeld  modular unit.}

Let $\tau$ be a coordinate function on $\Omega$ and let
$\Lambda_{\tau} = <1,\tau>$ be the  rank  two  free $A$-submodule of
$\CI$ generated by $1$  and $\tau$. Consider the following product
\[
e_{\Lambda_{\tau}}(z)= z \prod_{\lambda \in \Lambda_{\tau}
\backslash \{0\}} \left(1-\frac{z}{\lambda}\right) = z
\displaystyle{\prod_{a,b \in A\atop (a,b) \neq (0,0)}}
\left(1-\frac{z}{a \tau +b}\right).
\]
This product converges to give an entire, $\F_{q}$-linear,
surjective, $\Lambda_{\tau}$-periodic function
$e_{\Lambda_{\tau}}:\CI \rightarrow  \CI$ called the Carlitz
exponential function attached to  $\Lambda_\tau$. This is the
function field analogue of the classical $\wp$-function and it
provides the structure of a Drinfeld $A$-module to the additive
group scheme $\CI/\Lambda_{\tau}$.

The discriminant function $\Delta:\Omega \rightarrow \CI$ is the
analytic function defined by
\[
\Delta(\tau) = \mathop{\prod_{\alpha,\beta \in
T^{-1}A/A}}_{(\alpha,\beta) \neq  (0,0)} e_{\Lambda_{\tau}}(\alpha z
+ \beta).
\]
This is a modular form of weight $q^2-1$. For $I\neq 1$ a monic
polynomial in $A$ , let  $\Delta_I$ be the function
\begin{equation}\label{dmu}
\Delta_{I}(\tau) := \prod_{d|I\atop d\text{ monic }}
\Delta((I/d)\tau)^{\mu(d)}.
\end{equation}
Here $\mu(\cdot)$ denotes the M\"obius function on $A$. Since
$$\sum_{d|I\atop d\text{ monic }} \mu(d)=0$$
one has that the weight of $\Delta_I$ is $0$ so it is in fact a
$\Gamma_0(I)$-invariant function on $\Omega$. As we will see later
in equation \eqref{drindiv}, it is a modular unit, that is, its
divisor is supported on the cusps and defined over $K$ . We call
this the Drinfeld modular unit for $\Gamma_0(I)$.

\subsubsection{The Kronecker Limit Formula.}

The classical Kronecker limit formula links the Eisenstein series to
the logarithm  of the discriminant function $\Delta$. In this
section, we prove an analogue of this result in the function field
case.

We first compute the constant term $a_0(v)$ in the Taylor expansion
of $E(v,s)=E_1(v,s)$ around $s=1$. We have
\begin{equation}
E(v,s)=\frac{a_{-1}}{s-1}+a_0(v)+a_1(v)(s-1)+\ldots
\label{eisensteinone}
\end{equation}

where $a_{-1}$ is a constant independent of $v$. To compute
explicitly the coefficient function $a_0(v)$ we differentiate `with
respect to $v$', namely we apply the $\partial$ operator defined in
section~3.4.4 and then evaluate the result at $s=1$. This
computation gives
\[
\partial E(\cdot,s)|_{s=1}=\partial a_0 (\cdot).
\]
It follows from \eqref{log} that
\[
a_0(v)=\int_{v_0}^{v} \partial E(\cdot,s)(e)|_{s=1} d\mu(e)+C
\]
where $v_0$ is any vertex on the tree and $C$ is a constant. For
definiteness we can choose $v_0$ to be the vertex corresponding to
the lattice $[\OO_{\infty} \oplus \OO_{\infty}]$. The integration is
to be understood as the weighted sum  of the value of the function
on the edges lying on the unique path joining $v_0$ and $v$.

The function $\partial E(\cdot,s)$ on the edges in $Y(\T_0)$ is
related to the logarithmic derivative of the discriminant function
$\Delta$ through an improper Eisenstein series studied by Gekeler in
\cite{geke2}.

We first define Gekeler's series \cite{geke1}. For $e=e(k,u) \in
Y(\T)$, $s\in\C$, let $\psi^s(e)= \text{sgn}(e) q^{-ks}$. Consider
the following Eisenstein series
\[
F(e,s)=\sum_{\gamma \in \Gamma_{\infty}\backslash \Gamma}
\psi^s(\gamma(e)).
\]
This series converges for $Re(s)\gg 0$. Let $m,n\in A$, $(m,n) \neq
(0,0)$. For $\omega = \text{ord}_\infty(mu+n)$ let
\[
\psi^s_{m,n}(e) = \psi_{m,n}^{s}(e(k,u)) =
\psi^s_{m,n}\left(\begin{pmatrix} \pi^k & u
\\ 0  & 1  \end{pmatrix}\right) =
\begin{cases}  -q^{(k-2\deg(m)-1)s}   &  if\;\;\  \omega   \geq  k-\deg(m)\\
q^{(2\omega-k)s} & if \;\; \omega < k-\deg(m) \end{cases}
\]
we have
\[
F(e,s)=\psi^s(e)+ \sum_{{m \in A\atop m\text{ monic }}}\sum_{n\in
A\atop (m,n)=1} \psi^{s}_{m,n}(e).
\]

One can consider the limit as $s\rightarrow 1$ but the resulting
function $F(e,1)$, while $G(A)$ invariant, is not harmonic. The
series for $F(e,1)$ does not converge.

Gekeler \cite{geke1}(4.4) defines a conditionally convergent
improper Eisenstein series $\tilde{F}(e)$ as follows.
\begin{equation}
\tilde{F}(e)=\sum_{m \in A \atop \text{monic}} \sum_{n \in A \atop
(m,n)=1} \psi_{m,n}(e) + \psi(e) \label{impeis}
\end{equation}
This function is harmonic, but not $G(A)$ invariant. The relation
between these two functions is given as follows
\cite{geke1}(Corollary 7.11)
\begin{equation}
\label{releis} F(e,1)=\tilde{F}(e)-\sgn(e)\frac{q+1}{2q}
\end{equation}

This equation helps us relate the functions $E(\cdot,1)$ and
$\log|\Delta|$ as they are connected to $F$ and $\tilde{F}$ through
their derivatives. Precisely, we have the following

\begin{thm}[Gekeler] Let $\tilde{F}$ be as above. We have
\begin{equation}
\partial \log |\Delta|(e)=(1-q) \tilde{F}(e)
\label{gekeis}
\end{equation}
\end{thm}
\begin{proof} See \cite{geke2},
Corollary~2.8.\end{proof}

Moreover the  following lemma describes a relation between the
series $E(v,s)$ and $F(e,s)$.

\begin{lem}\label{rela}
Let $E(v,s)$ be the series in equation \eqref{eisensteinone} and let
$F(e,s)$ be the series defined in equation \eqref{releis}. Then
\begin{equation}
\partial E(\cdot,s)(e)=(q^s-1)F(e,s).
\label{partial}
\end{equation}
\end{lem}

\begin{proof}
It follows from the definition of the derivative of a function given
in \eqref{log} that
\[
\partial E(\cdot,s)(e) = E(t(e),s) - E(o(e), s).
\]
Set $e=e(k,u)=\vec{v(k,u)v(k-1,u)}$. Let $\phi^s_{m,n}(v)$ be the
function defined in \eqref{thephi}. There are four cases to
consider.\vspace{.05in}

\noindent {\bf Case 0.} For $e=e(k,u)$
\begin{align*}
\phi^s(t(e))-\phi^s(o(e))
&= q^{(k-1)s}-q^{-ks}  \\
&= (q^{s}-1)q^{-ks} \nonumber\\
&= (q^s- 1)\psi^s(e) .\nonumber
\end{align*}

\noindent {\bf Case 1.} If $\omega > k-1-\deg(m)$, then
\begin{align*}
\phi^s_{m,n}(t(e))-\phi^s_{m,n}(o(e))
&= q^{(k-1-2\deg(m))s}-q^{(k-2\deg(m))s}  \\
&= (1-q^{s})q^{(k-1-2\deg(m))s} \nonumber\\
&= (q^s- 1)\psi^s_{m,n}(e).\nonumber
\end{align*}

\noindent {\bf Case 2.} If $\omega < k-1-\deg(m)$, then
\begin{align*}
\phi^s_{m,n}(t(e))-\phi^s_{m,n}(o(e)) &=
q^{(2\omega-(k-1))s}-q^{(2\omega-k)s}\\
&= (q^s-1)(q^{(2\omega-k)s}\nonumber \\
&= (q^s-1)\psi^s_{m,n}(e).\nonumber
\end{align*}

\noindent {\bf Case 3.} If $\omega= k-1-\deg(m)$, so
$2\omega-k=2k-2-2\deg(m)$, then
\begin{align*}
\phi^s_{m,n}(t(e))-\phi^s_{m,n}(o(e))
&= q^{(k-2\deg(m)-1)s}-q^{(2\omega -k)s}\\
&= q^{(k-1-2\deg(m))s}-q^{(2k-2-2\deg(m))s}\nonumber \\
&= (q^s-1)(q^{(2\omega-k)s}\nonumber\\
&= (q^s-1)\psi^s_{m,n}(e).\nonumber
\end{align*}
These computations show that
\[
\partial E(\cdot,s)(e)=(q^s-1)F(e,s).
\]
\end{proof}\vspace{.05in}

Taking the limit as $s\rightarrow 1$ we have
$$\partial E(\cdot,1)(e)=(q-1)F(e,1)$$
Let $\Lambda(v,s)=-L_{\infty}(s)E(v,s)=\frac{q}{q-1} E(v,s)$ be the
function defined in Theorem~\ref{funeqv}. We have the following

\begin{thm}[``Kronecker's First Limit Formula'']
The function $\Lambda(v,s)$ has an expansion around $s=1$ of the
form
\begin{equation}
\Lambda(v,s)=\frac{b_{-1}}{s-1}+   \frac{q}{1-q}
\log|\Delta|(v)-\frac{q-1}{2} \log |\cdot|_i(v) + C +
b_1(v)(s-1)+\ldots \label{klf1}
\end{equation}
where $b_{-1}$ and $C$ are constants independent of $v$.
\end{thm}

\begin{proof}

It follows from Lemma~\ref{rela} that
\[
\partial \Lambda(\cdot,1)(e)=q F(e,1).
\]
Using \eqref{releis} we obtain
\[
\partial\Lambda(\cdot,1)(e)=q(\tilde{F}(e)-\frac{q+1}{2q} \sgn(e))
\]

To obtain the constant term in the Laurent expansion we integrate
the right hand side of this expression from $v_0$ to $v$. From
\eqref{gekeis} we have that the first term is
$$q\tilde{F}(e)=\frac{q}{1-q} \partial log|\Delta|(e)$$
so its integral is $\log|\Delta|(v)-\log|\Delta|(v_0)$. The integral
of the second term is
$$\int_{v_0}^{v} \frac{q+1}{2q} \sgn(e) d\mu(e)=\frac{q+1}{2q}
k(v)=\frac{q+1}{2q}(-\log |\cdot|_i(v))$$
where $|\cdot|_i$ is the `imaginary part', namely the distance from
$K_{\infty}$  of any element $\tau_v$ in $\lambda^{-1}(v)$, which
descends to a function on the tree. This follows from the discussion
on pp 371-372 of \cite{geke1}.

Combining these two expressions we get the theorem.
\end{proof}

Observe that each term that appears here is analogous to a term
which appears in the classical Kronecker Limit formula.

\subsection{A special value of the $L$-function}

Using the functional equation for $\Lambda(v,s)$ as stated in
Theorem~\ref{funeqv} and the expansion \eqref{klf1}, we obtain the
following result

\begin{thm}\label{spval} Let $f$ and $g$ be two newforms of level $I_1$ and $I_2$
respectively with $I_1$ and $I_2$ relatively prime polynomials in
$A$. Let $I=I_1I_2$. Then
\begin{equation*}
\Phi_{f,g}(0)= \frac{q}{1-q} \int_{Y(\T_0(I)}
\log|\Delta_I|\delta(f,g)
\end{equation*}
Here, $\log|\Delta_I|$ is interpreted as a function on $Y(\T)$ by
$\log|\Delta_I|(e)=\log|\Delta_I|(o(e))$ and $\delta(f,g)$ is as in
\ref{petersson}.
\end{thm}

\begin{proof}
From the description of $\Phi(s)$ given in \eqref{thefunct1} we have
\[
\Phi(0)= \lim_{s \rightarrow 0} \sum_{d|I\atop d\text{ monic }}
\frac{\mu(d)}{|d|^{s}}\int_{Y(\T_0(I))}
\Lambda((I/d)e,s)\delta(f,g).
\]
From the functional equation of $\Lambda(e,s)$  and the Limit
Formula \eqref{klf1}, we have
\begin{equation*}
\Lambda(e,s)=\frac{b_{-1}}{s}+\frac{q}{1-q}\log|\Delta|(e)
-\frac{q^2-1}{2} \log |\cdot|_i(e) + C + \text{h.o.t.(s)}
\end{equation*}
where $C$ is a constant and h.o.t.(s) denotes higher order terms in
$s$. Since
$$\sum_{d|I\atop d\text{ monic }} \mu(d)=0 \text{ and } <f,g>=0$$
we have
$$\sum_{d|I\atop d\text{ monic }}
\Lambda((I/d)e,0)=\frac{q}{1-q}\log|\Delta_I(e)|$$
as the sum of the residues of the poles is $0$ , the constant term
$C$  gets multiplied by $0$ and finally
$$\sum_{d|I\atop d\text{ monic }} \mu(d)\log|\cdot|_i((I/d)e)=0$$
as $|(I/d)e|_i=|(I/d)||e|_i$ and one can easily check that
$\prod_{d|I\atop d\text{ monic }} (I/d)^{\mu(d)}=(1)$, so its log is
$0$. It follows that
\begin{equation*}
\Phi_{f,g}(0)=\Phi(0)=-\frac{q}{q-1} \int_{Y(\T_0(I))}
\log|\Delta_I|(o(e))\delta(f,g).
\end{equation*}
\end{proof}

Since $\delta(f,g)$ is orientation invariant, we can replace the
usual integration on edges by integration over positively oriented
edges to get
\begin{equation}
\Phi_{f,g}(0)=\Phi(0)=-\frac{q}{q-1} \int_{Y^{+}(\T_0(I))} \left(
\log|\Delta_I|(o(e)) +  \log|\Delta_I|(t(e)) \right)
\delta^{+}(f,g). \label{specialvalue}
\end{equation}
Here $\delta^{+}(f,g)=f(e)g(e)d\mu^{+}(e)$ and $\mu^{+}$ denotes the
Haar measure on the positively oriented edges:
$\mu^+(e)=\frac{q-1}{|Stab_{e}(\Gamma_{\infty})|}$. \vspace{.1in}

\section{Elements in $K$-theory}
\label{ktheory}

\subsection{The group $H^3_{\M}(X,\Q(2))$}

Let $X$ be an algebraic surface over a field $F$. The second graded
piece of the Adams  filtration  on $K_1(X)\otimes \Q$ is usually
denoted by $H^3_{\M}(X,\Q(2))$. It has the following description in
terms of generators and relations.

The elements of this group are represented by finite formal sums
\[
\sum_{i} (C_i,f_i)
\]
where $C_i$ are curves on $X$  and $f_i$ are $F$-valued rational
functions on $C_i$ satisfying the cocycle condition
\begin{equation}\label{coco}
\sum_i \div(f_i)=0.
\end{equation}

Relations  in   this  group   are  given  by   the  tame   symbol of
functions. Precisely, if $C$ is a  curve on $X$ and  $f$ and $g$ are
two functions on $X$,  the tame symbol of $f$ and $g$ at $C$ is
defined by
\[
T_C(f,g) = (-1)^{\ord(g) \ord(f)}
\frac{f^{\ord(g)}}{g^{\ord(f)}},\qquad \ord(\cdot) = \ord_C(\cdot).
\]

Elements of the form $\displaystyle{\sum_C} (C, T_C(f,g))$ are said
to be  zero in  $H^3_{\M}(X,\Q(2))$.

\subsection{The regulator map on surfaces}

We define the regulator as the boundary map in a localization
sequence. We use the formalism of Consani \cite{cons1}.

Let $\Lambda$ be  a henselian discrete valuation  ring with fraction
field $F$ and let $X$ be a smooth, proper surface defined over $F$.
We set $\bar{X}=X \times Spec(\bar{F})$ for $\bar{F}$ an algebraic
closure of $F$. By a {\em semi-stable model} of $X$ (or semi-stable
fibration) we mean a flat, proper morphism $\XX \to Spec(\Lambda)$
of finite type over $\Lambda$, with generic fibre $\XX_{\eta} \cong
X$ and special fibre $\XX_{\nu}=Y$, a reduced divisor with normal
crossings in $\XX$. $\eta$ and $\nu$ denote respectively the generic
and closed points of $Spec(\Lambda)$. The scheme $\XX$ is assumed to
be non-singular and the residue field at $v$ is assumed to be
finite.

The scheme $Y$ is a finite union of irreducible components:
$Y=\cup_{i=1}^{r} Y_i$, with $Y_i$  smooth, proper, irreducible
surfaces. Let $J$ be a subset of $\{1,2,\ldots,r\}$ whose
cardinality is denoted by $|J|$. We set $Y_J=\cap_{j \in J} Y_j$ and
define
$$Y^{(j)}=\begin{cases} \XX & \text{ if } j=0 \\ \displaystyle{\coprod_{|J|=j}} Y_{J}& \text{
if } 1 \leq j \leq 3 \\
\emptyset & \text { if } j>3. \end{cases}$$
Let $\iota:Y \rightarrow \XX$ denote the subscheme inclusion map.
$\iota$ induces a push-forward homomorphism $\iota_\ast:
CH_1(Y^{(1)}) \to CH_1(\XX)$ and a pullback map $\iota^\ast:
CH^2(\XX) \to CH^2(Y^{(1)})$. Let $J=\{j_1,j_2\}$, with $j_1 < j_2$
and $I = J-\{j_t\}$, for $t\in\{1,2\}$. Then, the inclusions
$\delta_t: Y_J \rightarrow Y_{I}$ induce push-forward maps
$\delta_{t\ast}$ on the Chow homology groups. The Gysin morphism
$\gamma: CH_1(Y^{(2)}) \to CH_1(Y^{(1)})$ is defined by $\gamma =
\sum_{t=1}^2(-1)^{t-1}\delta_{t\ast}$.

Let
$$PCH^1(Y)=\frac{ker [\iota^*{\iota}_{*}:CH_{1}(Y^{(1)})
\rightarrow CH^2(Y^{(1)})]}{im [\gamma:CH_{1}(Y^{(2)}) \rightarrow
CH_1(Y^{(1)})]}/\{torsion\}$$
and define the $\nu$-adic Deligne cohomology to be
$$H^3_{\D}(X_{/\nu},\Q(2))=PCH^1(Y)\otimes \Q$$
If certain `standard conjectures' are satisfied, it follows from
Theorem~3.5 of \cite{cons1} that
$$\dim_{\Q} H^3_{\D}(X_{/\nu},\Q(2))=-{\ord_{s=1}} L_{\nu}(H^2(\bar{X},\Q_\ell),s),\qquad s\in\C$$
where $L_{\nu}(H^2(\bar{X}),s)$ is the local Euler factor at ${\nu}$
of the Hasse-Weil $L$-function of $X$.

There is a localization sequence which relates the motivic
cohomology of $\XX,X$ and $Y=\XX_{\nu}$ \cite{bloc}. When
specialized to our case it is as follows:
$$\dots \longrightarrow  H^3_{\M}(\XX,\Q(2))\longrightarrow H^3_{\M}(X,\Q(2))
\stackrel{\partial}{\longrightarrow} H^3_{\D}(X_{/\nu},\Q(2))
\longrightarrow H^4_{\M}(\XX,\Q(2))\longrightarrow \dots
$$
The $\nu$-adic regulator map $r_{\D,\nu}$ is defined to be the
boundary map $\partial$ in this localization sequence. If $\sum_i
(C_i,f_i)$ is an element of $H^3_{\M}(X,\Q(2))$ then
$$r_{\D, \nu}\left(\sum_i (C_i,f_i)\right)=\sum_i \div(\bar{f_i})$$
where $\bar{f_i}$ is the function $f_i$ extended to the Zariski
closure of $C$ in $\XX$. The condition $\sum \div(f_i)=0$ shows that
the `horizontal' divisors cancel each other out and so the image of
the regulator map is supported in the special fibre $\XX_v$.

Explicitly, one has the following formula for the regulator
\begin{equation}
\label{explicitregulator} r_{\D, \nu}\left(\sum_i (C_i,f_i)\right)=\sum_i
\sum_{Y} \ord_Y(f_i) Y
\end{equation}
where $Y$ runs through the components of the reduction of the
Zariski closure of the curves $C_i$.

This regulator map clearly depends on the choice of model. However,
Consani's work shows that the dimension of the target space does not
depend on the model since the local $L$-factor does not. Since the regulator map
is simply the boundary map of a localization sequence it satisfies the expected functoriality 
properties. 

While all our calculations below are with respect to a particular model,
perhaps the correct framework to work with the non-Archimedean
Arakelov theory of Bloch-Gillet-Soule \cite{bgs}.

\subsection{The case of products of Drinfeld modular curves}
\label{semistablefibre}
We now apply the results of the previous section to the case of the
self product of a Drinfeld modular curve $X_0(I)$ and the prime
$\infty$. In \cite{teit} page 280, Teitelbaum describes how to construct a
model $\XX_0(I)$ of the curve $X_0(I)$ over $\OO_{\infty}$. This
model has semi-stable reduction at $\infty$ and he describes a covering by affinoids
which have a canonical reduction. The special fibres of  the affinoids covering 
$\XX_0(I)$ are made up of two types of  components --  either of the type $(T_i \cup T_j)$ 
where the $T_i$ and  $T_j$ are isomorphic to $\CP^1_{\F_q}$ with all but one rational point deleted and 
meet at that point $T_{ij}=T_i \cap T_j$, or of the form $T_i$ where $T_i$ is isomorphic to $\CP^1_{\F_q}$ 
with all but one rational point deleted. 

The self-product $\XX_0(I) \times \XX_0(I)$ has, therefore, a covering by products of the affinoids covering $\XX_0(I)$ so there are four possibilities for the special fibre :
\begin{itemize}
\item (i)   $(T_1 \cup T_2) \times T_3$ 
\item  (ii) $T_1 \times (T_3 \cup T_4) $
\item (iii) $T_1 \times T_3$
\item  (iv) $(T_1 \cup T_2) \times (T_3 \cup T_4)$
\end{itemize} 
depending on whether the reduction is of the first or second type above. Therefore it is made up of components of the form 
$$T_1 \times T_4 \hspace{1in} T_1\times T_3$$
$$T_2 \times T_4 \hspace{1in} T_2 \times T_3$$
One has the following schematic representation --
\[
\xy
(0,0)*{}; (20,0)*{} **\dir{-};
(20,0)*{};(20,20)*{}**\dir{-};
(20,20)*{};(0,20)*{}**\dir{-};
(0,20)*{};(0,0)*{}**\dir{-};
(0,0)*{}; (-20,0)*{} **\dir{-};
(-20,0)*{};(-20,20)*{}**\dir{-};
(-20,20)*{};(0,20)*{}**\dir{-};
(20,0)*{};(20,-20)*{}**\dir{-};
(20,-20)*{};(0,-20)*{}**\dir{-};
(0,-20)*{};(0,0)*{}**\dir{-};
(-20,0)*{}; (-20,-20)*{} **\dir{-};
(-20,-20)*{};(0,-20)*{}**\dir{-};
(10,10)*{T_1 \times T_3};
(10,-10)*{T_2 \times T_3};
(-10,10)*{T_1 \times T_4};
(-10,-10)*{T_2 \times T_4};
(2,2)*{{\mathbf P}};
(0,0)*{\bullet};
(20,20)*{\circ};
(20,-20)*{\circ};
(-20,20)*{\circ};
(-20,-20)*{\circ};
(0,20)*{\circ};
(0,-20)*{\circ};
(-20,0)*{\circ};
(20,0)*{\circ};
\endxy
\]
$$\mathrm {Figure \;\;1.}$$
This reduction, however is not semi-stable. In the first three cases there is no problem but in case (iv) above there are four components and  all of them meet at the point ${\bf P}=(T_{12},T_{34})$, hence it is not semi-stable. 

However, if we blow up $\XX_0(I) \times \XX_0(I)$ at this point the
special fibre of the blow-up  is locally  normal crossings. Locally
the special fibre consists of five  components, $Y_1,\dots, Y_5$,
where
$$Y_3=\widetilde{T_1 \times T_4} \hspace{1in} Y_1=\widetilde{T_1 \times T_3}$$
$$Y_4=\widetilde{T_2 \times T_4} \hspace{1in} Y_2=\widetilde{T_2 \times T_3}$$
are the strict transforms of the components $T_i \times T_j$ above
and $Y_5\simeq \CP^1 \times \CP^1$ is the exceptional fibre
\cite{cons2}, (Lemma 4.1). 

We label it in this curious way as it is important in what follows. If one
thinks of the point of intersection as the origin, then $Y_1$ is the
strict transform of the  first quadrant, $Y_2$ of the one  below it,
$Y_3$ of the quadrant to the left of $Y_1$ and $Y_4$ the strict
transform of the quadrant to the left of $Y_2$. The diagram below is a schematic representation of the situation --

\[
\xy
(-10,0)*{};(0,10)*{}**\dir{-};
(0,10)*{};(10,0)*{}**\dir{-};
(10,0)*{};(0,-10)*{}**\dir{-};
(0,-10)*{};(-10,0)*{}**\dir{-};
(0,10)*{};(0,20)*{}**\dir{-};
(0,20)*{};(-10,30)*{}**\dir{-};
(-10,30)*{};(-20,30)*{}**\dir{-};
(-20,30)*{};(-30,20)*{}**\dir{-};
(-30,20)*{};(-30,10)*{}**\dir{-};
(-30,10)*{};(-20,0)*{}**\dir{-};
(-20,0)*{};(-10,0)*{}**\dir{-};
(-20,0)*{};(-30,-10)*{}**\dir{-};
(-30,-10)*{};(-30,-20)*{}**\dir{-};
(-30,-20)*{};(-20,-30)*{}**\dir{-};
(-20,-30)*{};(-10,-30)*{}**\dir{-};
(-10,-30)*{};(0,-20)*{}**\dir{-};
(0,-20)*{};(0,-10)*{}**\dir{-};
(0,10)*{};(0,20)*{}**\dir{-};
(0,20)*{};(10,30)*{}**\dir{-};
(10,30)*{};(20,30)*{}**\dir{-};
(20,30)*{};(30,20)*{}**\dir{-};
(30,20)*{};(30,10)*{}**\dir{-};
(30,10)*{};(20,0)*{}**\dir{-};
(20,0)*{};(10,0)*{}**\dir{-};
(20,0)*{};(30,-10)*{}**\dir{-};
(30,-10)*{};(30,-20)*{}**\dir{-};
(30,-20)*{};(20,-30)*{}**\dir{-};
(20,-30)*{};(10,-30)*{}**\dir{-};
(10,-30)*{};(0,-20)*{}**\dir{-};
(0,-20)*{};(0,-10)*{}**\dir{-};
(30,10)*{};(40,0)*{}**\dir{.};
(30,-10)*{};(40,0)*{}**\dir{.};
(-30,10)*{};(-40,0)*{}**\dir{.};
(-30,-10)*{};(-40,0)*{}**\dir{.};
(10,30)*{};(0,40)*{}**\dir{.};
(-10,30)*{};(0,40)*{}**\dir{.};
(10,-30)*{};(0,-40)*{}**\dir{.};
(-10,-30)*{};(0,-40)*{}**\dir{.};
(30,20)*{};(40,30)*{}**\dir{.};
(20,30)*{};(30,40)*{}**\dir{.};
(-30,20)*{};(-40,30)*{}**\dir{.};
(-20,30)*{};(-30,40)*{}**\dir{.};
(30,-20)*{};(40,-30)*{}**\dir{.};
(20,-30)*{};(30,-40)*{}**\dir{.};
(-30,-20)*{};(-40,-30)*{}**\dir{.};
(-20,-30)*{};(-30,-40)*{}**\dir{.};
(15,15)*{Y_1};
(15,-15)*{Y_2};
(-15,15)*{Y_3};
(-15,-15)*{Y_4};
(0,0)*{Y_5};
(9,6)*{Y_{15}};
(9,-6)*{Y_{25}};
(-9,6)*{Y_{35}};
(-9,-6)*{Y_{45}};
(3,15)*{Y_{13}};
(3,-15)*{Y_{24}};
(15,3)*{Y_{12}};
(-15,3)*{Y_{34}};
\endxy
\]
$$\mathrm{Figure \;\;2.}$$
Recall that this is the local picture -- to obtain the semi-stable model  we have to repeat this procedure
for every point of intersection of the components $T_i \times T_j$ - namely
at the points denoted by $\circ$  in {\rm Figure 1}.
 So the components of special fibre consist of the the strict
transforms of the $T_i \times T_j$ with all the four corners being
blown up. 

Observe that the labeling $Y_i$  above is with respect to which
corner of the $T_i \times T_j$ is being considered -- so for example $T_i \times T_j$ will be  labelled $Y_1$ if the South-Western corner is blown up but will be labelled $Y_4$ if the North-Eastern corner is blown up, $Y_2$ if the North-Western corner and $Y_3$ if the South-Eastern corner is blown up. 

Let $Y_{ij}$ denote the cycle $Y_{i} \cap Y_j$, if it exists. For example, one has  cycles $Y_{15}, Y_{12}, Y_{13}$ but no cycle $Y_{14}$ as $Y_1$ and $Y_4$ do not intersect. Similarly, let $Y_{ijk}$ denote the cycle     $Y_i \cap Y_j \cap Y_k$, if it exists. From the diagram one can see that for any such cycle at least one of the $i, j$ or $k$ has to be $5$, say $k=5$ and the cycle is $Y_{ij5}= Y_{i5} \cap Y_{j5}$. Further, one does not have cycles $Y_{145}$ and $Y_{235}$. Since the cycles $Y_{i5}$ and $Y_{j5}$ are rulings on $Y_5 \simeq \CP^1 \times \CP^1$ their intersection number is either $1$ or $0$ and so the cycles $Y_{ij5}$ have support on one point with multiplicity one. 

 In the group $H^3_{\D}((X_0(I) \times X_0(I))_{/\infty},\Q(2))$ one has
cycles coming from the restriction of the generic cycles as well as
certain cycles supported in the exceptional fibres. Locally, the
restriction of horizontal and vertical components give the cycles
$Y_{12}+Y_{34}+(Y_{15}-Y_{45}+Y_{25}-Y_{35})$ and
$Y_{13}+Y_{24}+(Y_{15}-Y_{45}-Y_{25}+Y_{35})$ \cite{cons2}, (Lemma
4.1). In the exceptional fibre $Y_5$ over $\mathbf{P}$ one also has
the cycle $\Z_{\bf{P}}=Y_{15}+Y_{45}-Y_{25}-Y_{35}$. Computing the
intersection with the other cycles show that this is not the
restriction of a generic cycle -- in fact, it is orthogonal to them and the cycles $Y_{12}, Y_{13}, Y_{24}$ and $Y_{34}$ as well. 

There are relations in this group coming from the image of the Gysin
map $\gamma$. For example, the difference of the image of the cycles
$Y_{15}$ in $CH^1(Y_1)$ and $CH^1(Y_5)$ lies in the image of the
Gysin map, so is $0$ in $H^3_{\D}((X_0(I) \times
X_0(I))_{/\infty},\Q(2))$. So there is a well defined $Y_{15}$ in $H^3_{\D}((X_0(I) \times
X_0(I))_{/\infty},\Q(2))$.  Similarly, the cycles $Y_{ij}$, which lie in both  $Y_i$ and $Y_j,  i, j \in \{1,\dots 5\}$, are well defined. Further,  the cycle which is $Y_{12}$ with respect to ${\bf P}$ is $Y_{34}$ 
with respect to the point ${\bf P'}$ to the right of ${\bf P}$ and so is counted only once in 
$H^3_{\D}((X_0(I) \times X_0(I))_{/\infty},\Q(2))$, and similarly for the others. So the local cycles described above  coming from  the restriction of the horizontal and vertical cycles patch up to give global cycles in  $H^3_{\D}((X_0(I) \times X_0(I))_{/\infty},\Q(2))$.

\subsubsection{A description in terms of the graph.}

Using the relation between the Bruhat-Tits tree and the special
fibre described at the end of Section \ref{purecover} or in \cite{teit} one can also
express this local picture in terms of the graph. Recall that components of the special fibre of $\Omega$ correspond to vertices on the tree and two components intersect at an edge. From that we have that the graph $\T_0(I)$ consists of a finite graph $\T_0(I)^0$ and finitely many ends. $\T_0(I)^0$ is the dual graph of the intersection graph 
of the special fibre of $\XX_0(I)$.  The situation where the canonical reduction of an affinoid has 
two components corresponds to an edge $e$ with vertices $o(e)$ and $t(e)$, both of which are $\T_0(I)^0$.
The situation when the canonical reduction has only one component corresponds to an edge $e$ with a distinguished 
vertex which is in  $\T_0(I)^0$. 

The special fibre of the product then has the following local description - it corresponds to either two, one or four pairs of vertices depending on whether we have case (i) or (ii), (iii) or (iv) above. In case (iv), the four pairs of vertices correspond to the four pairs of components and  the point ${\mathbf P}=(T_{12},T_{34})$ corresponds to a pair of
edges $(e_{12},e_{34})$ .  So we can re-label the cycle $\Z_{{\mathbf P}}$ as
$\Z_{(e_{12},e_{34})}$ where $(e_{12},e_{34})$ is the point being blown up.

The regulator of an element supported on curves uniformized by the
Drinfeld upper half plane lying on $X_0(I) \times X_0(I)$ can also
be expressed in terms of the graph. Since components in the special
fibre correspond to vertices of the graph on can rewrite the
regulator in terms of vertices. Let $Y_v$ denote the component
corresponding to a vertex $v$. From the definition of $\log|\cdot|$ one
has $\ord_{Y_v}(f)=\log|f|(v)$. So one can rewrite the expression
\eqref{explicitregulator} as
\begin{equation}
r_{\D,\infty}\left(\sum_i (C_i,f_i)\right)=\sum_i \sum_{v} \log|f_i|(v) Y_v
\label{logregulator}
\end{equation}
where $v$ runs through the vertices of the Bruhat-Tits graphs of
$C_i$. In Section \ref{theelem}, the element we construct will be
supported on curves isomorphic to $X_0(I)$ so we can express its
regulator using \eqref{logregulator}.

\subsection{The special cycle $\Z_{f,g}$}

As mentioned before, the motivic cohomology group of the surface
$X_0(I) \times X_0(I)$ can be decomposed with respect to eigenspaces
for pairs of cusp forms $(f,g)$ and this results in a decomposition
of the $\infty$-adic Deligne cohomology group as well. We denote
these groups by $H^3_{\D}(h^1(M_f) \otimes
h^1(M_g)_{/\infty},\Q(2))$. The local $L$-factor at $\infty$
\eqref{lfactorinfty} and Consani's theorem \cite{cons1}, Theorem
3.5, shows that this space is $1$ dimensional.

There is a special cycle in this group which plays the role played
by the $(1,1)$-form
$$\omega_{f,g}=f(z_1)\overline{g(z_2)}(dz_1\otimes d\bar{z}_2- d\bar{z}_1\otimes dz_2)$$
in the classical case. While $\omega_{f,g}$ is not represented by an
algebraic cycle, in our case there is a special cycle, supported in
the special fibre, which represents it. It is defined as follows.
For $f,g$ two cuspidal automorphic forms of JLD type and
$\Z_{(e,e')}$ as above, we define
$$\Z_{f,g} = \sum_{e,e' \in Y(\T_0(I)^0)} f(e)\overline{g(e')} \Z_{(e,e')}$$
in $H^3_{\D}(h^1(M_f) \otimes h^1(M_g)_{/\infty},\Q(2))$.  The action of the Hecke correspondence 
is through its action on $f$ and $g$ and so that shows  that this cycle lies in the $(f,g)$
component with respect to the Hecke action.

Note this this cycle is orientation invariant as
$(\bar{e},\bar{e}')=(e,e')$ and
$f(\bar{e})\bar{g}(\bar{e}')=f(e)\bar{g}(e')$. Further, as it is
composed of the cycles $\Z_{(e,e')}$ it is orthogonal to the cycles
which come by restriction from the generic Neron-Severi group.

\subsection{A special element in the motivic cohomology group}

In this section we  will use the Drinfeld   modular  unit $\Delta_I$
defined in \eqref{dmu} on  the diagonal $D_0(I)$ of $X_0(I)$ to
construct a canonical element $\Xi_0(I)$ in the motivic cohomology
$H^3_\M(X_0(I) \times X_0(I),\Q(2))$ of the self-product $X$ of the
Drinfeld curve $X_0(I)$. The trick is to `cancel out' the zeroes and
the poles of (a power of) $\Delta_I$ using certain functions
supported on the vertical and horizontal fibres of $X$. The
existence of these functions is a consequence of the function field
analogue of the Manin-Drinfeld theorem proved by Gekeler in
\cite{geke2}. Theorem ~\ref{thethm} provides a more explicit
description of them. As a corollary, we get an effective version of
the Manin-Drinfeld theorem in the function field case.

\subsubsection{Cusps.}

Let $I\in A$ be a monic, square-free polynomial. We first compute
the divisor of the function $\Delta_I$ explicitly. For this  we need
to work with an explicit description of the set of the cusps of
$X_0(I)$. It is well known that the set of these points is in
bijection with the set
\[
\Gamma_0(I)\backslash\Gamma/\Gamma_\infty
\stackrel{\simeq}{\to}\{[a:d]~:~d~|I,~a \in
(A/tA)^{*},~t=(d,I/d),~~a,d ~\text{monic, coprime}\}/\mathbb
F_q^\ast.
\]
We will denote the cusp corresponding to $[a:d]$ by $P^{a}_{d}$.
Since $I$ is  square-free, the cusps are of the form $P_d=P^1_{d}$,
where $d$ is a monic divisor of $I$. For a function $F = F(\tau)$ on
$\Omega$ and $f \in A$ let $F(f)$ denote the function $F(f\tau)$.
For $a,b \in A$, let $(a,b) = \text{g.c.d}\{a,b\}$ and  $[a,b] =
\text{l.c.m}\{a,b\}$. As $A$ is a P.I.D. they are both elements of
$A$. If $J$ is an element of $A$, the symbol $|J|$ denotes the
cardinality of the set $A/(J)$, where $(J)$ is the ideal generated
by $J$.

\begin{lem}\label{thelem}
Let $I \in A$ be square-free and monic  and assume that $I'$ and $d$
are monic divisors of $I$. Then
\begin{equation*}
\ord_{P_d} \Delta(I') = |I|~\frac{|(d,I')|}{|[d,I']|}
\end{equation*}
where the order at a cusp is computed in terms of a local
uniformizer as in \cite{gere} section~2.7.
\end{lem}

\begin{proof}
It follows from \cite{geke2} section~3 that
\[
\ord_{P_d}\Delta= |(I/d)|,\qquad \ord_{P_d} \Delta(I)= |d|.
\]
To obtain an explicit description of the divisor of $\Delta(I')$ on
$X_0(I)$ we need to compute the ramification index of $P_d$ over
$P_{d'}$, where $d' = \text{g.c.d}\{d,I'\}$. It follows from {\it
op.cit}, Lemma~3.8 that
\[
\text{ram}^{P_d}_{P_{d'}} = \frac{ |I||(d,I')|}{|d||I'|}.
\]
Therefore, one gets
\begin{align*}
\ord_{P_d} \Delta(I') &=  \text{ram}^{P_d}_{P_{d'}}\cdot\ord_{
P_{d'}} \Delta(I') =
\frac{|I||(d,I')|}{|d||I'|} \cdot |(d,I')| \\
&=   |I| \frac{|(d,I')|}{|[d,I']|}.
\end{align*}

\end{proof}

It follows from the definition of the function $\Delta_I$ in
\eqref{dmu} and  Lemma~\ref{thelem} that

\begin{equation}
 \div(\Delta_I)=\sum_{d|I\atop d\text{ monic }}
\mu(d)\div(\Delta(I/d))=\prod_{f|I \atop f \text{ prime
}}(1-|f|)(\sum_{d|I\atop d\text{ monic }} \mu(d) P_d).
\label{drindiv}
\end{equation}
A  simple modular unit is a (Drinfeld) modular unit whose divisor is
of the form  $k(P-Q)$,  where  $P$  and $Q$ are cusps of $X_0(I)$
and $k\in \mathbb Z$. The following theorem shows that there exists
$\kappa \in \mathbb N$ such that the function $\Delta_I^{\kappa}$
can be decomposed into a product of such units.\vspace{.05in}

\begin{thm}
\label{thethm} Let $I$ be a square-free, monic element of $A$ and
let $I=\prod_{i=0}^{r} f_i$ be the prime factorization of $I$, with
the $f_i's$ monic elements of $A$. Let $\kappa=\prod_{i=0}^{r}
(1+|f_i|)$. Then
\[
\Delta_I^{\kappa}=\prod_{a|(I/f_0)\atop a\text{ monic }} F_{a}
\]
where the functions $F_a$  are simple modular units and
\[
\div(F_a)=\prod_{i=0}^{r}(1-|f_i|^2) \mu(a)(P_a-P_{f_0 a}).
\]
\end{thm}

\begin{proof} The proof will follow from the following lemmas.

\begin{lem}
\label{lemmada} Let $P_a$ be a cusp of $X_0(I)$. Then, the divisor
of the form
$$D_a=\prod_{d|I\atop d\text{ monic }} \Delta(d)^{\mu(I/d)
|I|\frac{|(a,d)|}{|[a,d]|}}$$
is
$$\div(D_a)=\prod_{f|I\;\atop f\;\text{ monic },\;\text{ prime }} (1-|f|^2) \mu(a) P_a.$$
\end{lem}
\begin{proof} Let $P_b$ be a cusp of $X_0(I)$. Then, it follows from
lemma~\ref{thelem} that
$$\ord_{P_b} (D_a)=\sum_{d|I\atop d\text{ monic }}
\mu(I/d)|I|^2\frac{|(a,d)(b,d)|}{|[a,d][b,d]|}.$$ \vspace{.05in}

We consider the following cases

\noindent {\bf Case 1 ($a\neq b$)}. In this case there is a prime
element $f\in A$ dividing $a$ but not $b$ (or vice versa). Assume
that $f|a$ and $\text{g.c.d.}\{f,b\}=1$. Then
$$\ord_{P_b}(D_a)= \sum_{d|(I/f)\atop d\text{ monic }}
\mu(I/d)|I|^2\left(\frac{|(a,d)(b,d)|}{|[a,d][b,d]|} -
\frac{|(a,fd)(b,fd)|}{|[a,fd][b,fd]|} \right ).$$
Since $f|a$, $(a,fd)=f(a,d)$ and $[a,fd]=[a,d]$. Further, $(f,b)=1$,
$(b,fd)=(b,d)$ and $[b,fd]=f[b,d]$. Therefore
$$\frac{|(a,d)(b,d)|}{|[a,d][b,d]|} -
\frac{|(a,fd)(b,fd)|}{|[a,fd][b,fd]|}=0.$$
so we have $\ord_{P_b}(D_a)=0$.

\noindent{\bf Case 2 ($a=b$)}. In this case we have to show that
\begin{equation}\label{form}
\sum_{d|I\atop d\text{ monic }}
\mu(I/d)|I|^2\frac{|(a,d)|^2}{|[a,d]|^2}=\mu(a)\prod_{f|I\atop
f\text{ monic, prime }}(1-|f|^2).
\end{equation}
The proof is by induction on $a$. If $a=1$, then the left hand side
of \eqref{form} is
$$\sum_{d|I\atop d\text{ monic }} \mu(I/d)|I|^2\frac{|(1,d)|^2}{|[1,d]|^2}=
\sum_{d|I\atop d\text{ monic }}
\mu(I/d)\left(\frac{|I|}{|d|}\right)^2=\prod_{f|I\atop f\text{
monic, prime }}(1-|f|^2)$$
and the lemma follows.

Now, we assume that \eqref{form} holds for some $a|I$. Let $f$ be a
monic prime of $A$ such that $f|I$ and $(f,a)=1$.  We will show that
\eqref{form} holds for $fa$. The left hand side of \eqref{form} is
now
$$\sum_{d|I\atop d\text{ monic }} \mu(I/d)|I|^2\frac{|(fa,d)|^2}{|[fa,d]|^2}=\sum_{d|(I/f)\atop d\text{ monic }}
\mu(I/d)|I|^2
(\frac{|(fa,d)|^2}{|[fa,d]|^2}-\frac{|(fa,fd)|^2}{|[fa,fd]|^2}).$$
If $d|(I/f)$, we have  $(fa,d)=(a,d)$, $[fa,d]=f[a,d]$ and
$(fa,fd)=f(a,d)$,$[fa,fd]=(f)[a,d]$. So
$$\sum_{d|I\atop d\text{ monic }} \mu(I/d)|I|^2\frac{|(fa,d)|^2}{|[fa,d]|^2}=\sum_{d|(I/f)\atop d\text{ monic }}
\mu(I/d)|I|^2
\left(\frac{1}{|f|^2}-1\right)\left(\frac{|(a,d)|^2}{|[a,d]^2}\right)=$$
$$=-(1-|f|^2) \sum_{d|(I/f)\atop d\text{ monic }}\mu(I/fd)|(I/f)|^2
\frac{|(a,d)|^2}{|[a,d]|^2}.$$
By induction, this is
$$-(1-|f|^2)\mu(a)\prod_{g|(I/f)\atop g\text{ monic, prime }}(1-|g|^2)=\mu(fa)\prod_{g|I\atop g\text{ monic, prime }}(1-|g|^2).$$
This concludes the proof of the lemma.
\end{proof}\vspace{.1in}

Let $f_0$ be a prime element of $A$ dividing $I$. For $a|(I/f_0)$,
we set
$$F_a=D_a D_{f_0a}$$
where the functions $D_a$ and $D_{f_0a}$ are defined as in
lemma~\ref{lemmada}.
Then, by applying that lemma we have
$$\div(F_a)=\prod_{f|I\atop f \text{ monic, prime }}
(1-|f|^2)\mu(a)(P_a-P_{f_0a}).$$
So $F_a$ is a simple modular unit. The statements of the theorem
will follow by applying the next lemma

\begin{lem}\label{lastlemma} Under the same hypotheses of Theorem~\ref{thethm}, we
have

$$\prod_{a|(I/f_0)\atop a\text{ monic }} F_a=\prod_{d|(I/f_0)\atop d\text{ monic }}\prod_{a|(I/f_0)\atop a\text{ monic }}  \left(
\frac{\Delta(d)}{\Delta(f_0d)} \right)^{\mu(I/d)|I|
(1+\frac{1}{|f_0|})\frac{|(a,d)|}{|[a,d]|}}.$$
\end{lem}

\begin{proof}

From the definition of $F_a$ we have
\[
\prod_{a|(I/f_0)\atop a\text{ monic }} F_a=\prod_{a|(I/f_0)\atop
a\text{ monic }} \prod_{d|I\atop d\text{ monic }}
\Delta(d)^{\mu(I/d) |I| \left( \frac{|(a,d)|}{|[a,d]|} +
\frac{|(f_0a,d)|}{|[f_0a,d]|} \right)}.
\]
If $(d,f_0)=1$, then $(f_0a,d) = (a,d)$ and $[f_0a,d] = f_0[a,d]$.
So we get
$$ \frac{|(a,d)|}{|[a,d]|}+ \frac{ |(f_0 a,d)|}{|[f_0 a,d]|}=
 \frac{|(a,d)|}{|[a,d]|} \left( 1+\frac{1}{|f_0|} \right) = \frac{|(a,f_0d)|}{|[a,f_0d]|}+ \frac{ |(f_0 a,f_0d)|}{|[f_0 a,f_0d]|}.$$
Collecting together the terms with the same $d$, we obtain
$$\prod_{a|(I/f_0)\atop a\text{ monic }} F_a= \prod_{d|(I/f_0)\atop d\text{ monic }}
\left( \frac{\Delta(d)}{\Delta(f_0d)} \right)^{\mu(I/d)|I|
(1+\frac{1}{|f_0|})\left(\sum_{a|(I/f_0)\atop a\text{ monic }}
\frac{|(a,d)|}{|[a,d]|} \right) }.$$
Using an induction argument similar to the one used in the proof of
Lemma~\ref{lemmada}, we have
$$\sum_{a|(I/f_0)\atop a\text{ monic }} \frac{|(a,d)|}{|[a,d]|} =\prod_{f|(I/f_0)\atop f\text{ prime }}
(1+{\frac{1}{|f|}}).$$
\end{proof}

To finish the proof of the theorem, we notice that
$$\Delta_I=\prod_{d|I\atop d\text{ monic }} \Delta(d)^{\mu(I/d)}=\prod_{d|(I/f_0)\atop d\text{ monic }}
\left( \frac{\Delta(d)}{\Delta(f_0d)} \right) ^{\mu(I/d)}.$$
Let $\kappa=\prod_{i=0}^{r} (1+|f_i|)$. Then, it follows from
lemma~\ref{lastlemma} that
\begin{displaymath}
\prod_{a|(I/f_0)\atop a\text{ monic }} F_a= \Delta_I^\kappa.
\end{displaymath}
\end{proof}

As a corollary of Lemma~\ref{lemmada}, we obtain the following
result of independent interest.\vspace{.05in}

\begin{cor}[Effective Manin-Drinfeld theorem]

Let $I=\prod_{i=0}^{r} f_i$ be the monic, prime factorization of a
square free, monic polynomial $I$ in $A$. Then, the cuspidal divisor
class group is finite and its order divides
$\prod_{i=0}^{r}(1-|f_{i}|^{2})$.
\end{cor}

\begin{proof}
If $a$ and $a'$ are two cusps of $X_0(I)$, then it follows from
lemma~\ref{lemmada} that the function
$$F_{a,a'}=\frac{D_a}{ D_{a'}^{\frac{\mu(a)}{\mu(a')}}}$$
has divisor
$$\div(F_{a,a'}) = \prod_{i=0}^{r}(1-|f_i|^2)
\mu(a)(P_a-P_{a'}).$$
\end{proof}
\vspace{.1in}

\subsection{An element in $H^3_{\M}(X_0(I) \times
X_0(I),\Q(2))$}\label{theelem}

Using the factorization in Theorem \ref{thethm}, we can construct an
element of the motivic cohomology group as follows:

Let $D_0(I)$  denote the diagonal on $X_0(I)\times X_0(I)$ and let
$I=\prod_{i=0}^{r} f_i$ be the monic prime factorization of  $I$.
Let $\kappa=\prod_{i=0}^{r} (1+|f_i|)$. Let $F_d = D_dD_{f_0d}$ as
in Lemma~\ref{lemmada}. Consider the element
\begin{equation}\label{el}
\Xi_0(I) = (D_0(I),\Delta_I^{\kappa})-\left( \sum_{d|(I/f_0)} (P_{d}
\times X_0(I),P_d \times F_d) + (X_0(I) \times P_{f_0d}, F_{d}
\times P_{f_0d}) \right).
\end{equation}
It follows from Theorem~\ref{thethm} that this element satisfies the
cocycle condition \eqref{coco}, as the sum of the divisors of the
functions is a sum of multiples of terms of the form
\[  (P_d,P_d)-(P_{f_0d},P_{f_0d})  -  (P_d,P_d)  +  (P_{d},P_{f_0d})  +
(P_{f_0d},P_{f_0d}) - (P_{d},P_{f_0d}).
\]
Hence $\Xi_0(I)$ determines an element of $H^3_{\M}(X_0(I) \times
X_0(I),\Q(2))$.

\subsubsection{The regulator of $\Xi_0(I)$.}

From the formula give in \eqref{logregulator}, the regulator of our
element of $H^3_{\M}(X_0(I) \times X_0(I),\Q(2))$ is given by the
formula
\begin{equation}
r_{\D,\infty}(\Xi_0(I))= \sum_{v \in X(D_0(I))}
\log|\Delta_I^{\kappa}|(v)Y_{v} \label{regform}
\end{equation}
$$
 + \sum_{d|(I/f_0)} \left( \sum_{v
\in X((P_{d} \times X_0(I)))} \log|P_d \times F_d|(v)Y_v  + \sum_{v
\in X((X_0(I) \times P_{f_0d}))} \log|F_{d} \times P_{f_0d}|(v)Y_v
\right)$$

\subsection{ The final result}

We have the following theorem which relates the special value of the
L-function with the intersection pairing of certain cycles. This
intersection pairing is the intersection pairing on the group
$PCH^1(Y)$ obtained as the sum of the intersection pairings on the
Chow groups of the components. It is well defined as it vanishes on
the image of the Gysin map.

\begin{thm}
 Let $f$ and $g$ be Hecke eigenforms for $\Gamma_0(I)$ and
$\Phi_{f,g}$ the completed Rankin-Selberg $L$-function. Then one has
\begin{equation}
\Phi_{f,g}(0)=\frac{q}{ 2 (q-1)  \kappa} (r_{\D,\infty}
(\Xi_0(I)),\Z_{f,g})
\end{equation}
where $\Xi_0(I)$ is the element of the higher chow group constructed
above, $r_{\D,\infty}$ is the regulator map and $\Z_{f,g}$ is the
special cycle described above.\label{mainthm}
\end{thm}

\begin{proof}
We first compute the pairing of the regulator of $\Xi_0(I)$ with
$\Z_{f,g}$. For this we have to compute the pairing of special fibre of the total
transform of the diagonal $D_0(I)$ with $\Z_{f,g}$ as well as the
pairing of the vertical and horizontal components with $\Z_{f,g}$. Since the pairing
is the sum of all the pairings of the components one can compute it
locally - around a point ${\bf P}=(e,e')$ which is being blown up as in Section \ref{semistablefibre} .   

Recall that $\Z_{\BP}=Y_{15}+Y_{45}-Y_{25}-Y_{35}$. We have the following intersection numbers of $\Z_{{\BP}}$ with the various cycles $Y_{ij}$ --

\begin{itemize}
\item  $(\Z_{\BP},Y_{12})=(\Z_{\BP},Y_{13})=(\Z_{\BP},Y_{24})=(\Z_{\BP},Y_{34})=0$
\item $(\Z_{\BP},Y_{15})=(\Z_{\BP},Y_{45})=-2$
\item $(\Z_{\BP},Y_{25})=(\Z_{\BP},Y_{35})=2$
\end{itemize}
These can easily be computed using the fact $\Z_{\BP}$ is the difference of rulings on $Y_{5}$. 

Locally, $D_0(I)$ is the blow-up of the diagonal in $(T_1 \cup T_3) \times (T_2 \cup T_4)$, where $T_i$ are as in Section \ref{semistablefibre}. The part of the diagonal which passes through ${\BP}$ is the sum of the diagonals in $T_1 \times T_3$ and $T_2 \times T_4$. Let $\Delta_1$ and $\Delta_4$ denote the strict transforms of these diagonals in $Y_1$ and $Y_4$. The total transform is 
$$\Delta_1+Y_{15}+\Delta_4+Y_{45}$$
as the blow up of the diagonal in $T_1 \times T_3$ has exceptional fibre $Y_{15}$ and similarly for the other diagonal.  One has $(\Z_{\BP},\Delta_i)=0$  since $\Z_{\BP}$ is supported in the exceptional fibre.  

For vertical or horizontal components the total transform is \cite{cons2}, Lemma 4.1,
$$Y_{13}+(Y_{15}-Y_{35}) + Y_{24}+(Y_{25}-Y_{45})$$
and
$$Y_{12}+(Y_{15}-Y_{25})+Y_{34}+(Y_{35}-Y_{45})$$
respectively. Hence, using the intersection numbers computed above, we have
\begin{itemize}
\item $(\Z_{\BP},Y_{13}+(Y_{15}-Y_{35}) + Y_{24}+(Y_{25}-Y_{45}))=0$
\item $(\Z_{\BP},Y_{12}+(Y_{15}-Y_{25})+Y_{34}+(Y_{35}-Y_{45}))=0$
\end{itemize}
The regulator of $\Xi_0(I)$ is 
$$\sum_{v \in X(D_0(I))} \log|\Delta_I^{\kappa}|(v)Y_{v} + $$
$$+ \sum_{d|(I/f_0)} \left( \sum_{v \in X((P_{d} \times X_0(I)))}
\log|P_d \times F_d|(v)Y_v  + \sum_{v \in X((X_0(I) \times
P_{f_0d}))} \log|F_{d} \times P_{f_0d}|(v)Y_v \right).$$
From above we can see that the vertical and horizontal components have intersection number $0$ with $\Z_{f,g}$, so it suffices to compute the intersection number of the diagonal component of the regulator with $\Z_{f,g}$.

Locally, at the picture corresponding to the point $(e,e)$, the diagonal components 
appear with  multiplicities 
\begin{itemize}
\item $\kappa \log|\Delta_0(I)|(o(e))$  for  $\Delta_4$ and $Y_{45}$
\item $\kappa \log|\Delta_0(I)(t(e))$ for $\Delta_1$ and $Y_{15}$.
\end{itemize}
as the vertex $o(e)$ corresponds to the component $\Delta_4$ and the vertex $t(e)$
corresponds to the component $\Delta_1$ of the diagonal. Hence the diagonal component is a sum of terms of the type  
$$\kappa \log|\Delta_0(I)(o(e)) (\Delta_4+Y_{45}) + \kappa \log|\Delta_0(I)(t(e)) (\Delta_1+Y_{15}).$$
Using the fact that $t(e)=o(\bar{e})$  and the calculations above, we get --
$$(r_{\D,\infty}(\Xi_0(I)),\Z_{f,g})=(-2\kappa)\int_{e \in Y^{+}_0(I)}
\left( \log|\Delta_I|(o(e))+\log|\Delta_I|(t(e))\right) f(e)g(e)
d\mu^{+}(e).$$
This is a finite sum as $f$ and $g$ have finite support. 

Comparing this with \eqref{specialvalue} gives us our final result.

\begin{equation}
\Phi_{f,g}(0)=\frac{q}{ 2 (q-1)  \kappa} (r_{\D,\infty}
(\Xi_0(I)),\Z_{f,g}) \label{finalresult}
\end{equation}

As $\Phi_{f,q}(s-1)=\Lambda(h^1(M_f) \otimes h^1(M_g),s)$, we get
Theorem \ref{mainthm}.

\end{proof}

\subsubsection{An application to elliptic curves.}

Theorem~\ref{mainthm} provides some evidence for
Conjecture~\ref{conj} in the case of a product of two non-isogenous
elliptic curves over $K$.

If $E$ is a non-isotrivial (that is,  $j_E \notin \F_q$),
semi-stable elliptic curve over $K$ with conductor
$I_{E}=I\cdot\infty$ and split  multiplicative reduction  at
$\infty$, by the work of Deligne \cite{deli}, Drinfeld, Zarhin and
eventually Gekeler-Reversat \cite{gere} we have that $E$ is modular.
This means that the Hasse-Weil $L$-function $L(E,s)$ is equal  to
the $L$-function of an automorphic  form $f$ of JLD-type with
rational fourier coefficients
\[
L(E,s)=L(f,s)=\sum_{\m \;\text{pos.div}} \frac{c(f,\m)}{|\m|^{s-1}}.
\]
Furthermore, there exists a Drinfeld modular curve $X_0(I)$ of level
$I$ and a dominant morphism (the modular parametrization)
\begin{equation}\label{param}
\pi_f:X_0(I) \longrightarrow E.
\end{equation}

Now, let  $E$   and  $E'$  be  two  such modular  elliptic curves
with corresponding  automorphic  forms $f$  and  $g$  of levels
$I_1$ and $I_2$. Assume  that $(I_1,I_2)=1$ and that $I=I_1 I_2$ is
square-free .  Then, the $L$-function of $H^2(\bar{E} \times
\bar{E}',\Q_\ell)$ can be expressed in terms  of the $L$-function of
the Rankin-Selberg convolution of $f$  and $g$. K\"unneth's theorem
gives the decomposition
\[
L(H^2(\bar{E}\times \bar{E}'),s) =
L(H^2(\bar{E}),s)^2L(H^1(\bar{E})\otimes
H^1(\bar{E}'),s)=\zeta_A(s-1)^2 L(H^1(\bar{E})\otimes
H^1(\bar{E}'),s).
\]
The completed $L$-function of $H^1(\bar{E})\otimes H^1(\bar{E})$ is
the function  $\Phi(s-1)=\Phi_{f,g}(s-1)$ of \eqref{thefunct1}. We
set
\begin{equation}\label{clf}
\Lambda_{E,E'}(s)=L_{\infty}(s-1)^2 \zeta_A(s-1)^2 \Phi(s-1).
\end{equation}
Then $\Lambda_{E,E'}(s)$ is the completed $L$-function of
$H^2(\bar{E} \times \bar{E}',\Q_\ell)$.\vspace{.1in}

The following result is an application of Theorem~\ref{mainthm}.
\begin{thm}\label{ourcor}
Let $E$ and $E'$ be elliptic curves over $K$ satisfying the above
conditions. Then, there is an element $\Xi \in H^3_{\M}(E\times
E',\Q(2))$ such that
\begin{equation}\label{spv2}
\Lambda_{E,E'}^*(1)=\frac{q \deg(\Pi)^2}{2 \kappa (1-q)^3
\log_e(q)^{2}} \left(r_{\D,\infty}(\Xi),\Z_{E,E'}\right)
\end{equation}
where $\Pi$ is the restriction of the product of the modular
parameterizations of $E$ and $E'$ to the diagonal $D_0(I)$ of
$X_0(I)$ and $\Lambda_{E,E'}^*(1)$ is the first non-zero value in
the Laurent expansion at $s=1$.
\end{thm}
\begin{proof}
Let $\pi_f \times \pi_g: X_0(I) \times X_0(I) \rightarrow E \times
E'$ be the product  of the modular parameterizations $\pi_f$ and
$\pi_g$. Let $\Xi = (\pi_f \times \pi_g)_*(\Xi_0(I)) \in H^3_{\M}(E
\times E',\Q(2))$ be the push-forward cocycle in motivic cohomology,
where $\Xi_0(I)$ is the class defined in \eqref{el}. Let $\Z_{E,E'}
= (\pi_f \times \pi_g)_*(\Z_{f,g})$ be the push-forward cycle in the
Chow group where $\Z_{f,g}$ is the 1-cycle considered in
theorem~\ref{mainthm}. The two push-forward maps contribute a factor
$\deg(\Pi)^2$ to the equation. Moreover, the residue at $s=1$ of the
archimedean factor in \eqref{clf} is ${\log_e(q)}^{-2}$. The result
then follows from Theorem~\ref{mainthm}.
\end{proof}\vspace{.1in}

For a self-product of elliptic curves of the type considered in
Theorem~\ref{ourcor}, part C. of Conjecture~\ref{conj} asserts that
$$\Lambda_{E,E'}^*(1)=\frac{|coker(R_{\D})|} {|ker(R_{\D})|} \log_e(q)^{-2}.$$
Note that \eqref{spv2} contains the correct power of $\log_e(q)$.
Further, one has that the intersection number $( r_{\D,\infty}(\Xi),
Z_{E,E'} )$ divides $coker(R_{\D})$. Finally, the power $(1-q)^3$ in
the denominator of \eqref{spv2} can be partly explained in terms of
the kernel of the regulator map $R_\D$. The group $H^3_{\M}(E \times
E',\Q(2))$ contains certain elements coming from $H^2_{\M}(E \times
E',\Q(1)) \otimes H^1_\M(E \times E',\Q(1))$ called decomposable
elements. Note that $H^2_{\M}(E \times E',\Q(1)) \cong Pic(E \times
E')$ and $H^1_{\M}(E \times E',\Q(1)) \cong K^*$. Elements of the
form $D \otimes u$, for $D \in NS(E \times E')$ and $u$ a torsion
element in $K^*$, belong to $ker(R_\D)$. There are $(q-1)$ elements
$u$ coming from $\F_{q}^{*}$ and there are two independent elements
$D$ of $NS(E \times E')$ providing $(q-1)^2$ such elements.

\section{ Final Remarks}

Many of the arguments can be carried out in much greater generality
- for example, the ground field could be any local field. The
assumption $(I_1,I_2) = 1$ in Theorem~\ref{ourcor} is not that
essential. Along the lines of the arguments in \cite{basr}, we can
prove a similar result under the weaker assumption that $I_1$ and
$I_2$ have some common factors, but are not identical.

As suggested by the referee,  another direction in which this work can be generalized is that of higher weight forms. Scholl generalized the work of Beilinson's for forms of weight $> 2$ -- however, while in our case the analogue of weight 2 forms are the $\Q_{\ell}$-valued  harmonic cochains on the tree, it is not clear to me what the analogue of higher weight forms is. One might expect that perhaps harmonic cochains with values in local systems might play the role. 

As remarked earlier, since all the factors appearing are analogous to factors appearing in the classical number 
field case it would be interesting to know if there was some common underlying field over which the conjecture can be formulated and proved, for which the above work and the classical theorems are special cases. 

\bibliographystyle{alpha}

\bibliography{DrinfeldReferences}

\noindent Ramesh Sreekantan\\
Indian Statistical Institute \\
$8^{th}$ Mile, Mysore Road	 \\
Jnana Bharathi \\ 
Bangalore, 560 059 India \\
\noindent Email: rameshsreekantan@gmail.com
\end{document}